\documentclass[11pt,a4paper,reqno]{article}

\usepackage{amsmath, amsthm, amssymb,mathtools}
\usepackage{enumerate}
\usepackage[inline]{enumitem}

\usepackage{parskip}
 
\usepackage{url}
\usepackage[latin1]{inputenc}
\usepackage{comment,hyperref,xcolor}
\hypersetup{colorlinks=true, linkcolor=black}

\usepackage{caption}
\usepackage{subcaption}

\numberwithin{equation}{section}

\DeclareMathOperator{\Li}{Li} 
\DeclareMathOperator{\Var}{Var} 
\DeclareMathOperator{\dist}{dist} 
\DeclareMathOperator{\support}{supp}

\newtheorem{theorem}{Theorem}

\newtheorem{lemma}[theorem]{Lemma}

\newtheorem{corollary}[theorem]{Corollary}

\newtheorem{proposition}[theorem]{Proposition} 

\newtheorem{definition}[theorem]{Definition}

\newtheorem{remark}[theorem]{Remark}
\newtheorem{remarks}[theorem]{Remarks} 
\numberwithin{theorem}{section}

\theoremstyle{remark} 

\newtheorem{example}[theorem]{Example}

\theoremstyle{definition}

\begin{document}

\title{Poisson Representable Processes} 
\author{Malin P. Forsstr\"om\thanks{Email: palo@chalmers.se Address: Mathematical Sciences, Chalmers University of Technology and University of Gothenburg, SE-412 96 G\"oteborg, Sweden}, Nina Gantert\thanks{Email: nina.gantert@ tum.de, Address: SoCIT, Department of Mathematics, 85748 Garching b. M\"unchen, Boltzmannstr. 3, Germany}, , Jeffrey E. Steif\thanks{Email: steif@chalmers.se,  Address: Mathematical Sciences, Chalmers University of Technology and University of Gothenburg, SE-412 96 G\"oteborg, Sweden}}

\maketitle

\begin{center}\itshape
    In celebration of Geoffrey Grimmett's 70th birthday
\end{center}

\begin{abstract}
    Motivated by Alain-Sol Sznitman's interlacement process, we consider the set of $\{0,1\}$-valued processes which can be constructed in an analogous way, namely as a union of sets coming from a  Poisson process on a collection of sets. Our main focus is to determine which processes are representable in this way. Some of our results are as follows. (1) All positively associated Markov chains and a large class of renewal processes are so representable. (2) Whether an average of two product measures, with close densities, on $n$ variables, is representable is related to the zeroes of the polylogarithm functions.
    (3) Using (2), we show that a number of tree-indexed Markov chains as well as the Ising model on~\( \mathbb{Z}^d ,\)~\( d\geq 2,\) for certain parameters are not so representable. (4) The collection of permutation invariant processes that are representable corresponds exactly to the set of infinitely divisible random variables on $[0,\infty]$ via a certain transformation. (5) The supercritical (low temperature) Curie-Weiss model is not representable for large~$n$.
\end{abstract}

\tableofcontents

\section{Introduction}

Let $S$ be a finite or countably infinite set and let $\nu$ be a $\sigma$-finite
measure on $\mathcal{P}(S)\backslash \{\emptyset\}$ where
$\mathcal{P}(S)$ is the power set of $S$. This generates a $\{0,1\}$-valued process
$X^\nu=\{X^\nu_i\}_{i\in S}$ defined as follows.

We first consider the Poisson process
$Y^\nu$ on $\mathcal{P}(S)\backslash \{\emptyset\}$ with intensity measure $\nu$ (see~\cite{lp} for the definition).
Note that $\mathcal{P}(S)\backslash \{\emptyset\}$ can be viewed as an open subset
of $\{0,1\}^S$ and hence has a nice topology and Borel structure.
$Y^\nu(\omega)$ is then a collection $\{B_j\}_{j \in I}$ of nonempty subsets (perhaps with repetitions) of $S$. Note that $| I |<\infty$ a.s.\ if  $\| \nu \| \coloneqq \nu(\mathcal{P}(S)\backslash\{\emptyset\}) <\infty$ and that
$|I |=\infty$ a.s.\ if  $\| \nu \| =\infty$. 

Finally, we define $\{X^\nu_i\}_{i\in S}$ by
$$
X^\nu_i= \begin{cases}
    1 &\text{if } i\in \cup_{j\in I} B_j \cr 
    0 &\text{otherwise.}
\end{cases}
$$
To see that $X^\nu_i$ is a random variable, one observes
that  $X^\nu_i=1$ if and only if
$$
Y^\nu\cap \mathcal{S}_i \neq \emptyset,
$$
where
$$
\mathcal{S}_i\coloneqq \{T\in \mathcal{P}(S)\colon  i\in T\}
$$
and the above is an event by definition of a Poisson process
since $\mathcal{S}_i$ is a open set in $\mathcal{P}(S)\backslash\{\emptyset\}$. Loosely speaking, \( X^\nu\) is obtained by taking the union of the sets arising in the Poisson process, and identifying this with the corresponding \( \{ 0,1\}\)-sequence.

\begin{definition}
    We let \( \mathcal{R}\) denote the set of all processes $(X_i)_{i\in S}$ which are equal (in distribution) to $X^\nu$ for some $\nu$. 
\end{definition}

Understanding which \( X\) are in \( \mathcal{R}\) seems to be an interesting question and will be the main focus of this paper.
We will for the most part deal with three different situations:
\begin{enumerate}
    \item $S$ is a finite set,
    \item $S$ is $\mathbb{Z}^d$ for some $d\ge 1$ and $\nu$ (and hence also $X^\nu$) is translation invariant under the natural  $\mathbb{Z}^d$-action on $\mathcal{P}(S)\backslash \{\emptyset\}$ and
    \item $S$ is infinite and $\nu$ (and hence also $X^\nu$) is invariant under all finite permutations of  $\mathcal{P}(S)\backslash \{\emptyset\}$.
\end{enumerate}
 
We give an alternative but equivalent description of this model in the case when $|S|<\infty$
which has a more combinatorial flavor. Let $S$ be a finite set and for each nonempty subset
$T$ of $S$, let $p(T)\in [0,1]$. Now, for each $\emptyset\neq T\subseteq S$, we 
independently ``choose"  $T$ with
probability  $p(T)\in [0,1]$ and we let $X$ be the union of the ``chosen" $T$'s which
we identify with a $\{0,1\}$-valued process 
$\{X_i\}_{i\in S}$. 
The correspondence between this formulation and the earlier one is that
$p(T)$ is simply the probability that a Poisson random variable with parameter $\nu(\{T\})$
is nonzero.

We introduce the following natural notation.
For $A\subseteq S$, we let
$$
\mathcal{S}_A^\cup \coloneqq \bigcup_{i\in A} \mathcal{S}_i\qquad 
\bigl(=\{T \in \mathcal{P}(S) \colon  A\cap T\neq \emptyset\} \bigr)
$$
and
$$
\mathcal{S}_A^\cap \coloneqq \bigcap_{i\in A} \mathcal{S}_i\qquad 
\bigl(=\{ T \in \mathcal{P}(S) \colon  A\subseteq T\} \bigr).
$$

We observe that for any $A\subseteq S$, we have that
$$P\bigl(X^\nu(A)\equiv 0\bigr)= e^{-\nu(\mathcal{S}_A^\cup)}.$$

We now discuss a number of examples. In these examples and throughout the rest of the paper, for \( p \in [0,1],\) we will let \( \Pi_p\) denote a product measure with 1's having density~\( p.\) 

\begin{example} 
    Let \( a>0, \) and let 
    $$
        \nu = \sum_{i\in S}  a \delta_{\{ i\}}.
    $$
    Then \( X^\nu \sim \Pi_{1-e^{-a}}.\) 
\end{example}
 
\begin{example}\label{example: 1.3} 
    Let \( a>0, \) and let \( \nu = a \delta_{S}.\) Then $X^\nu$ has distribution 
    $$
        (1-e^{-a}) \delta_{\mathbf{1}} + e^{-a} \delta_{\mathbf{0}},
    $$
    where $\mathbf{1}$ ($\mathbf{0}$)
    is the configuration consisting of all $1$'s ($0$'s). If \( S = \mathbb{Z}^d, \) then this yields a (trivial) non-ergodic process.
\end{example} 

\begin{example}
    Let \( S = \mathbb{Z},\) and let  
    $$
        \nu=\sum_{i\in \mathbb{Z}}  \delta_{\{i, i+1\}}.
    $$
    Then $X^\nu$ is the image of an i.i.d.\ sequence under a block (finite range)\color{black}{} map. More precisely, if  \( {(Y_i)_{i \in \mathbb{Z}} \sim \Pi_{1-e^{-1}},}\) then \( X^\nu \overset{d}{=} \bigl(\max(Y_i,Y_{i+1})\bigr)_{i \in \mathbb{Z}}.\)
\end{example}

\begin{example}
    Let \( S = \mathbb{Z},\) and let 
    $$
        \nu= \sum_{i\in \mathbb{Z}, \, n\ge 0} a_n \delta_{\{i, i+n\}}
    $$
    for some \( (a_n)_{n \geq 0}.\) 
    It is easy to see (similar to Example~\ref{example: 1.3}) that $X^\nu$ is a Bernoulli shift. This means that there is an equivariant map, i.e., a map commuting with the shift, from an i.i.d.\ process to the process in question. \color{black}{} However,  we don't know if it is a finitary factor of an i.i.d.\ sequence.
     Finitary means that the equivariant map above is a.s.{} continuous.  These concepts will play a very tiny role in this paper and not until Section~\ref{sec: stationary}. \color{black}
    The Borel-Cantelli Lemma immediately gives that $X^\nu\equiv 1$ a.s.\ if and only if $\sum_{n=0}^\infty a_n=\infty$.
\end{example}

\begin{example}
    The well-studied random interlacement process in $\mathbb{Z}^d$ for $d\ge 3,$ introduced by Alain-Sol Sznitman,
    falls into this context; see~\cite{DRS} for the
    definition and some of what is known. For those familiar with this, 
    we actually need to massage it slightly so  that it falls into our context.  In the random interlacement process, we have a Poisson process over random walk realizations modulo time shifts
    which are transient in both forward and backward time.
    If we now map such a trajectory (modulo time shifts) to
    its range and then take their union, we then obtain a process in \( \mathcal{R}\)  since the push-forward of a Poisson process is a Poisson process.  The $\nu$ in~\cite[Theorem 5.2]{DRS} would provide us with our $\nu$ (after pushing forward).
    The random interlacement process has an intensity parameter which just corresponds to scaling  the measure $\nu$. It was in fact this model which provided the motivation for our paper.
\end{example}

\begin{example}
    The union of the discrete loops that arise in a random walk loop soup  corresponds to a process in  \( \mathcal{R}.\)  
    The random walk loop soup is a well studied object in relation to the Brownian loop soup and  the discrete Gaussian free field. It was introduced in~\cite{LawlerTrujilloFerreras} and is defined in the following way: the rooted loop measure $\mu^{\rm RW}$ assigns to each (nearest neighbor) random walk loop in $\mathbb{Z}^2$ of length $2n$ the measure $(1/2n)4^{-2n}$ and the measure $\nu$ is given by $\lambda \mu^{\rm RW}$ where $\lambda \in (0, \infty)$ is an intensity parameter.
\end{example}

While we are limiting ourselves to countable sets, in continuous space, similar constructions (e.g., Boolean models, Poisson cylinder models) play a crucial role in stochastic geometry. We also want to mention that the idea of enriching a graph by attaching a Poissonian number of independent finite subgraphs comes up naturally in the analysis of random graphs, see for instance~\cite{bl} and~\cite{BJR}.

The paper is organized as follows.
Our main focus is to determine, for a given $\{0,1\}$-valued process, if it belongs to \(\mathcal{R}\). On the way to answering this question, we give some properties of processes in  \(\mathcal{R}\). It is easy to show that processes in  \(\mathcal{R}\) are positively associated. It turns out that they also have the so-called downward FKG property (but not necessarily the FKG property), see 
Theorem~\ref{theorem:DFK}. Also, if $X$ is a collection of i.i.d.\ $\{0,1\}$-valued random variables, then \(X \in \mathcal{R}\).
Taking \(S = \mathbb{Z}\), it is natural to ask if Markov chains or renewal processes are in  \(\mathcal{R}\). 
We show in Section~\ref{sec: Markov} that indeed all positively associated Markov chains are in \(\mathcal{R}\), and describe the corresponding measure $\nu$, see Theorem~\ref{theorem: Markov}. 
On the way to this result, we give a necessary and sufficient condition for a renewal process to be in \(\mathcal{R}\), see Theorem~\ref{theorem: general Markov}.
In Section~\ref{sec: inf perm}, we consider processes $X$ on $\{0,1\}^{\mathbb{N}}$ which are invariant under finite permutations, and we ask if they are in \(\mathcal{R}.\) 
In this section, we first prove a version of de Finetti's Theorem for possibly infinite measures (which turned out to be known) that is of independent interest, see Theorem~\ref{theorem: definetti}. 
%
%
We then show that the $X \in \mathcal{R}$ are in one-to-one correspondence with infinitely divisible distributions on \([0, \infty)\), see Theorem~\ref{theorem: characterization}. 
Interesting such examples include the $\{0,1\}$-sequences coming from classical urn models, see Examples~\ref{ex: 4.9} and~\ref{ex: 4.9 beta}.
In Section~\ref{sec: curie weiss}, we investigate the finite permutation invariant case. Taking a sequence $X_n \in \{0,1\}^n$, $n \geq 1$, such that each $X_n$ is permutation invariant, we show that in order to have $X_n \in \mathcal{R}$ for all $n\geq 1$, it is necessary that the arithmetic means of the $X_n$'s concentrate, see Theorem~\ref{theorem: variance and limit}. As an immediate consequence, we obtain that for temperatures less than the critical one, the Curie-Weiss model is not in \(\mathcal{R}\) for large \( n,\) see Theorem~\ref{theorem: curie weiss}. 
We then study finite averages of  $m$ product measures, and we give sufficent conditions for $X\in \mathcal{R}$ for $m \geq 2$ and for $X\notin \mathcal{R}$ for $m=2$, see Theorem~\ref{proposition: finite cases} for the latter case. The proof is quite technical, given that the statement is only about an average of two product measures, but Section~\ref{sec: tree} relies crucially on this result.
In Section~\ref{sec: tree}, we consider tree-indexed Markov chains on infinite trees and we give conditions on the parameters such that the process is not in \(\mathcal{R}\). With a similar argument, we show that the Ising model on $\mathbb{Z}^d$ for $d\ge 2$ is not in \(\mathcal{R}\) for a certain range of parameters. 
In Section~\ref{theorem: main ergodic}, we consider stationary processes $X^\nu$ on $\{0,1\}^{\mathbb{N}}$. We give a necessary and sufficient condition on $\nu$ such that $X^\nu$ is ergodic, see Theorem~\ref{theorem: main ergodic}, and a sufficent condition on $\nu$ such that $X^\nu$ is a Bernoulli shift, see Theorem~\ref{theorem: bernoulli shift}. We end in Section~\ref{sec: questions} with some open questions.

\section{Background definitions, and some first properties and examples}

We begin by recalling some basic 
definitions.

\begin{definition}
    A probability measure $\mu$ on $\{0,1\}^S$ is said to be  positively associated if for all increasing sets $A$ and $B$,
    $$
        \mu(A\cap B) \ge \mu(A)\mu(B).
    $$
\end{definition}

\begin{definition}
    A probability measure $\mu$ on $\{0,1\}^S$ is said to satisfy the FKG property if for all $I\subseteq S$ and all $\{a_i\}_{i\in I}$ with each
    $a_i\in \{0,1\}$ such that \( \mu( X_i=a_i \mbox{ for } i\in I)>0\)\color{black}, the conditional measure 
    $$
        \mu(\cdot \mid X_i=a_i \mbox{ for } i\in I)
    $$
    on $\{0,1\}^{S\backslash I}$ is positively associated.
\end{definition}
 
\begin{definition}
    A probability measure $\mu$ on $\{0,1\}^S$ is said to satisfy the downwards FKG property if in the definition of the FKG property, we only require the positive association of the conditional measure when each $a_i=0.$
\end{definition}

Clearly the FKG property implies the downward
FKG property which in turn implies positive
association. Note that while positive association and the FKG property are unaffected
by reversing 0's and 1's, this is not the case 
with the downward FKG property.

The following result demonstrates the very different
roles played by the \(0\)'s and \(1\)'s for our $X^\nu$.

\begin{theorem}\label{theorem:DFK}
    For every $S$ and $\nu$, $X^\nu$ has the downward FKG property. However, it does not necessarily satisfy the FKG property.
\end{theorem}

\begin{proof}
    We begin the proof by showing that $X^\nu$  has positive association for all $\nu$'s. In the case of finite $S$, it is immediate that $X^\nu$ is given by increasing functions of i.i.d.\ random variables and hence by Harris' Theorem has positive association. An easy approximation argument gives the result for general sets \( S.\)

    We now prove the downward FKG property.  To this end, recall that $Y^\nu$ is the Poisson process  on $\mathcal{P}(S)\backslash \{\emptyset\}$ with intensity measure $\nu.$ \color{black} When one conditions on the event that $X^\nu$ is zero on some subset $A\subseteq S$, one is conditioning on the event that no element in $\mathcal{S}_A^\cup$ occurred in $Y^\nu$. The conditional distribution of $Y^\nu$ then becomes $Y^{\nu|_{(\mathcal{S}_A^\cup)^c}}$ and hence the conditional distribution of $X^\nu$ is still of our form and hence is positively associated by the first part of this proof.

    Finally, we give an example of an \( X^\nu \) which does not have the FKG property.
Let ${S=\{1,2,3\}}$ and $\nu$ give weight $\log 2$ to each of $\{1,2\}$ and $\{2,3\}$. Then $Y^\nu$ is one of the following collections of sets, each having probability $1/4$. 
    \begin{enumerate}[label=(\alph*)]
        \item $\emptyset$,
        \item $\bigl\{\{1,2\}\bigr\}$,
        \item $\bigl\{\{2,3\} \bigr\}$, and
        \item $\bigl\{\{1,2\},\{2,3\}\bigr\}$.
    \end{enumerate} 
    If we condition on $x_2=1$, then we know $Y^\nu$ is one of the last three each then with conditional probability $1/3$. Now, it is immediate that conditioned on $x_2=1$, we have $x_1=1$ with probability $2/3$, $x_3=1$ with probability $2/3$ and $x_1=x_3=1$ with probability $1/3$, which is less than $4/9$. Hence $Y^\nu$ is not FKG.
\end{proof}

\begin{remark}
    The example in the proof of Theorem~\ref{theorem:DFK} also gives an example of an \( X \in \mathcal{R}\) such that \( X \mid X_2 = 1 \) is not in \( \mathcal{R}.\)
\end{remark}

\begin{remark}
    The upper invariant measure for the contact process exhibits similar behavior to the
  above example. It is not FKG (see~\cite{L})
  but it is downwards FKG (see~\cite{BHK}).
\end{remark}

\begin{proposition}
    Assume  $\nu$ is a translation invariant measure on \( \mathcal{P}(\mathbb{Z}) \smallsetminus \{ \emptyset \} \) that gives positive weight to an infinite subset $S$ which is not periodic. Then $X^\nu \equiv 1$ a.s.
\end{proposition}
 
\begin{proof}
    Since $S$ is not periodic, all of its translates are distinct and have the same $\nu$-weight. We show $X^\nu(0)= 1$ almost surely. Since $S$ is infinite, there are an infinite number of translates of $S$  containing 0 and so at least one of these will occur almost surely.
\end{proof}

The next proposition says that the overlap property of the sets that $\nu$ charges describes pairwise correlations in a simple way.

\begin{proposition}\label{prop: correlation}
    Assume that \( X = (X_s)_{s \in S} = X^\nu .\)  Then, for any \( k,\ell \in S, \) we have $$P(X_k=0,X_\ell=0)=P(X_k=0) P(X_\ell=0) e^{\nu(\mathcal{S}_{k}\cap \mathcal{S}_{\ell})}.$$
    More generally, if $A$ and $B$ are disjoint subsets of $S$, we have
     
    \begin{equation*}
    	P(X_A \equiv 0 ,\,  X_B \equiv 0)= 
    P(X_A \equiv 0) P(X_B \equiv 0)
    e^{\nu(\mathcal{S}_A^\cup \cap \mathcal{S}_B^\cup)}.
    \end{equation*}
    \color{black}
    
\end{proposition}

\begin{proof} We prove only the first statement. The second statement is proved in the same way. By definition,
    \begin{align*}
        &
        P(X_k=0,X_\ell=0)=e^{-\nu(\mathcal{S}_{k}\cup \mathcal{S}_{\ell})}=
        e^{-[\nu(\mathcal{S}_{k})+\nu(\mathcal{S}_\ell)-\nu(\mathcal{S}_{k}\cap \mathcal{S}_\ell)]}\\&\qquad=
        P(X_k=0) P(X_\ell=0) e^{\nu(\mathcal{S}_{k}\cap \mathcal{S}_\ell)}.
    \end{align*}
\end{proof}

\begin{corollary}
For any $S$ and $\nu$, if $X^\nu$ is pairwise independent, then it is an independent process.
\end{corollary}

\begin{proof}
	By inclusion-exclusion, it suffices to show that for any finite $A\subseteq S$, one has
	$$
		P(X^\nu_A \equiv 0)=
		\prod_{a\in A} P(X^\nu_a= 0).
	$$
	The assumption of pairwise independence together with Proposition~\ref{prop: correlation} implies that for $a,b\in A$, $a\neq b$, $\nu(\mathcal{S}_a\cap \mathcal{S}_b)=0$. 
	This yields
	$\nu(\mathcal{S}_A^\cup)=\sum_{a\in A}\nu(\mathcal{S}_a)$
	and hence
	$$
		P(X^\nu_A \equiv 0)=e^{-\nu(\mathcal{S}_A^\cup)}= e^{-\sum_{a\in A}\nu(\mathcal{S}_a)}=\prod_{a\in A} P(X^\nu_a= 0).
	$$
\end{proof}

The proof of the following useful lemma will be left
to the reader.

\begin{lemma}\label{lemma: 2.10new}
For any $S$, $\nu$ and $\nu'$, one has that
$$
X^{\nu+\nu'}\overset{d}{=} \max\{X^{\nu},X^{\nu'}\},
$$
where the latter two processes are assumed to be
independent.
\end{lemma}

\begin{lemma}\label{lemma: 2.10}
    Given $\{0,1\}$-valued random variables \( X=(X_k)_{k \in [n]}\), if there is a nonnegative measure \( \nu \) on \( \mathcal{P}([n])\backslash \{\emptyset \}\) such that 
    \[
    P\bigl(  X(K) \equiv 0 \bigr) = e^{-\nu(\mathcal{S}_K^\cup)},\quad K \subseteq [n],
    \]
    then \( X^\nu \overset{d}{=} X.\)
\end{lemma}

\begin{proof}
    If \( \nu \) is nonnegative, \( X^\nu\) exists. Since \( X\) and \( X^\nu\) agree on events of the form \( X_K \equiv 0,\) they must agree on all events, and hence the desired conclusion follows.
\end{proof}

\begin{lemma}\label{lemma: unique signed}
    Let  \( X = (X_1,X_2,\dots, X_n)\) be 
    $\{0,1\}$-valued random variables
    such that  \( P(X \equiv 0) > 0 \)\color{black}.
    Then there is a unique signed measure \( \nu \) on \( \mathcal{P}([n])\smallsetminus \{ \emptyset \}\) that satisfies 
    \begin{equation}\label{eq: step 1}
        \nu\bigl(\mathcal{S}_{I}^\cup\bigr) = -\log P\bigl( X(I) \equiv 0 \bigr),\quad I \subseteq [n]
    \end{equation} 
    and it is given by~\eqref{eq: Mobius inversion result}.
    (Note that by Lemma~\ref{lemma: 2.10}, if such a nonnegative measure \( \nu\) exists, then \( X = X^\nu\).)
\end{lemma}

\begin{proof}
    If a signed
    measure \( \nu \) exists then, since \( n \) is finite, we have
    \[
        \nu\bigl(\mathcal{S}_{I}^\cup\bigr)=\sum_{J \subseteq [n] \colon J \cap I \neq \emptyset}   \nu(J)
    \]
    and hence~\eqref{eq: step 1} is equivalent to the system of linear equations given by
    \begin{equation}\label{eq: step 2}
        \sum_{J \subseteq [n] \colon J \cap I \neq \emptyset}   \nu(J) =  -\log P\bigl( X(I) \equiv 0 \bigr),\quad I \subseteq [n].
    \end{equation} 
    The desired conclusion will thus follow if we can show that this system of linear equations always has a unique (possibly signed) solution \( (\nu(J))_{J \subseteq [n], J \neq \emptyset}.\) 
    To this end, we first rewrite~\eqref{eq: step 2} as
    \begin{align*}
        & 
        \sum_{\substack{J \subseteq [n]\smallsetminus I\colon \\ J \neq \emptyset}}   \nu(J)
        = 
         \sum_{\substack{J \subseteq [n]\colon \\ J \neq \emptyset }}   \nu(J) 
        -
        \sum_{J \subseteq [n] \colon J \cap I \neq \emptyset}   \nu(J) 
        \\&\qquad =
        -\log P\bigl( X([n]) \equiv 0 \bigr)
        -
        \Bigl( -\log P\bigl( X(I) \equiv 0 \bigr)\Bigr), \quad I\subseteq [n].
    \end{align*}
    Equivalently, this becomes
    \begin{align*}
        &
        \sum_{\substack{J \subseteq I \mathrlap{\colon} \\J \neq \emptyset}}   \nu(J)
        =
        -\log P\bigl( X([n]) \equiv 0 \bigr)
        +
        \Bigl( \log P\bigl( X([n]\smallsetminus I) \equiv 0 \bigr)\Bigr)
        , \quad I\subseteq [n].
    \end{align*}
    By the M\"obius inversion theorem, we see that this equation is equivalent to
    \begin{equation}\label{eq: Mobius inversion result}
        \begin{split}
        &\nu(K)
        = \sum_{I \subseteq K} (-1)^{|K|-|I|} 
        \Bigl(
        -\log P\bigl( X([n]) \equiv 0 \bigr)
        +
        \log P\bigl( X([n]\smallsetminus I) \equiv 0 \bigr)  \Bigr) 
        \\&\qquad
        = \sum_{I \subseteq K} (-1)^{|K|-|I|} 
         \log P\bigl( X([n]\smallsetminus I) \equiv 0 \bigr)  
        , \qquad \emptyset \neq K\subseteq [n].
        \end{split}
    \end{equation}
    This concludes the proof.
\end{proof}

\begin{lemma}
    Assume that \( X = X^\nu\) and that  \( (X_i)_{i \in [n]} \overset{d}{=} (X_{\sigma(i)})_{i \in [n]}\) for some \( \sigma \in S_n.\) Then \( \nu = \nu \circ \sigma.\)
\end{lemma}

\begin{proof}
    From Lemma~\ref{lemma: unique signed}, we know that if \( X = X^\nu\) exists, then $\nu$ is unique. If \( X = X^\nu\) then \( X = X_\sigma = X^{\nu \circ \sigma}, \) and hence we must have \( \nu = \nu\circ \sigma.\)
\end{proof}

The proof of the following lemma is left to the reader, the third part of whose proof uses a simple compactness argument.

\begin{lemma}\label{lemma: finite to infinite}
Consider a process $X=\{X_s\}_{s\in S}$. 
\begin{enumerate}[label=(\alph*)]

    \item If \( X=X^\nu\)  and \( B \subseteq S, \) then there is \( \nu_B \) such that \( X|_{B} = X^{\nu_B}.\) Moreover, for any non-empty measurable subset \( \mathcal{A} \subseteq \mathcal{P}(B), \) we have \( {\nu_B(\mathcal{A}) = \nu\bigl( \{ A' \in \mathcal{P}(S) \colon  A' \cap B\in \mathcal{A} \} \bigr).}\)\label{item: finite to infinite i}

    \item If \( X=X^\nu\) and \( B \subseteq S,\) then there is a measure \( \nu_{B,0} \) on \( \mathcal{P}(B)\smallsetminus \{ \emptyset \} \) such that \( X| \bigl\{ X(B^c)\equiv 0 \bigr\} = X^{\nu_{B,0}}.\) Moreover, \( \nu_{B,0} = \nu|_{\mathcal{P}(B)}. \) \label{item: finite to infinite iii}
    
    \item\label{item: finite to infinite ii} If there exist $S_1\subseteq S_2\subseteq \ldots$ such that $S=\bigcup_i S_i$ and  $X_{S_n} = X^{\nu_n}$ for some $\nu_n$, then $X=X^\nu$ for some $\nu$. (The projection of $\nu$ on to each $S_n$ will simply be  $\nu_n$.)
\end{enumerate} 
\end{lemma}

We first point out that when $n=2$, provided
we have positive association, $X$ is always of this form. In particular, one can check that
\( X = X^\nu,\) where for \( \emptyset\neq J \subseteq \{ 1,2 \}\) we set \( \nu(J) = -\log\bigl(1-p(J)\bigr) \) where
    \begin{equation*}
        \begin{cases}
            p(\{1,2\}) = 1 - P(X_1=0) P(X_2=0)/P(X_1=X_2=0) \cr 
            p(\{1\}) = 1-(P(X_1=X_2=0))/P(X_2=0) \cr 
            p(\{2\}) = 1-(P(X_1=X_2=0))/P(X_1=0).
        \end{cases}
    \end{equation*} 

    Sticking with $n=2$ in the nonpositively associated case, 
    it is interesting to see which sets the representing signed 
    measure $\nu$ gives negative weight to.

    \begin{example}
    Let \( X\) be \( (1,0) \) or \( (0,1) \) each with probability \((1-\varepsilon)/2 \) and equal to \( (0,0) \) or \( (1,1) \) each with probability \( \varepsilon/2,\) where \( \varepsilon<1/2.\) Then
    \[
    P \bigl( X(1)=0 \bigr) = P\bigl( X(2) = 0 \bigr) = 1/2
    \quad
    \text{and}
    \quad
    P\bigl( X(\{1,2\}) \equiv 0 \bigr) = \varepsilon /2.
    \]
    Consequently,
    \[
    \nu(\mathcal{S}_1\cup \mathcal{S}_2) = -\log \varepsilon/2 = \log 2 - \log \varepsilon 
    =
    2\nu\bigl( \{ 1 \} \bigr) + \nu \bigl( \{ 1,2 \} \bigr)
    \]
    and
    \[
    \nu(\mathcal{S}_1) = -\log 1/2 = \log 2 = \nu\bigl( \{ 1 \} \bigr) + \nu \bigl( \{ 1,2 \} \bigr)
    \]
    and hence
    \[
    \nu\bigl( \{ 1\} \bigr) =  - \log \varepsilon  >0
    \quad
    \text{and}
    \quad
    \nu\bigl( \{ 1,2\}\bigr) =   \log 2 + \log \varepsilon<0.
    \] 
\end{example}

When we now move to $n=3$, it is already
the case that positive association does not
imply that $X$ is of our form as the following example shows.

\begin{example}
    Choose \( \sigma \in S_3\) uniformly at random, and define \( X = (X_1,X_2,X_3) \) by \( X_j = \mathbf{1}({\sigma(j)=j}).\) Then
    \begin{equation*}
        \begin{cases}
            P(X \equiv 1) = 1/6 \cr 
            P\bigl(X = (1,1,0)\bigr) = P\bigl(X = (1,0,1) \bigr)= P\bigl(X = (0,1,1) \bigr) = 0 \cr 
            P\bigl(X = (1,0,0)\bigr) = P\bigl(X = (0,1,0) \bigr) = P\bigl(X = (0,0,1)\bigr) = 1/6 \cr 
            P(X \equiv 0) = 1/3.
        \end{cases}
    \end{equation*}
    The random vector \( X\) defined above is known to be positively associated (see, e.g., \cite{fds})\color{black}. On the other hand, one easily verifies that it is not of our form.
\end{example}

\begin{remarks}\mbox{}
\begin{enumerate}[label=(\roman*)] 
	\item More generally, Fishburn, Doyle and Shepp~(\cite{fds}) proved that if we choose \( \sigma \in S_n\) uniformly at random and define \( X = (X_1,X_2,\dots, X_n) \) by \( X_j = \mathbf{1}({\sigma(j)=j}), \) then  \( X \) is positively associated. \color{black}
    \item Jeff Kahn (\cite{Kahn}) proved the much stronger and much more difficult fact that the above random vector cannot be expressed as an increasing function of i.i.d.\ random variables. Nikita Gladkov~(\cite{g2023}) has extended this result further by showing that the above random vector cannot be obtained as the limit of increasing functions of i.i.d.\ random variables.
    \color{black}
    \item We will see another example later on of a positively associated process which is not in  \( \mathcal{R}\) for $n=3$; it will, in fact, be an average of two product measures.
\end{enumerate} 
\end{remarks}

We provide a further interesting example for
$n=3$ which is positively associated but
not in \( \mathcal{R}.\)

\begin{example}
    Let \( X,Y\) be i.i.d.\ 0 or 1 each with probability $1/2$. Consider \( (X,Y,XY).\)
    This is positively associated since the vector is given by increasing functions of i.i.d.\ random variables. Next,
if \( XY=0,\) then either \( X\) or \( Y \) is equal to zero. Consequently,
    \[
    P(X=Y=1 \mid XY=0) = 0
    \]
    and hence \( (X,Y,XY)\) is not downward FKG. By Theorem~\ref{theorem:DFK}, we conclude that this is
     not in \( \mathcal{R}.\)

    The unique signed measure \( \nu\) which satisfies \( X = X^\nu\) is given by
    \begin{equation*}
        \begin{cases}
            \nu \bigl( \{ 1 \} \bigr) = \nu \bigl( \{ 2 \} \bigr) = \log 2 \cr 
            \nu \bigl( \{ 3 \} \bigr) = \nu\bigl( \{1,3 \} \bigr) = \nu \bigl( \{2,3 \} \bigr)  = 0 \cr 
            \nu\bigl( \{ 1,2 \} \bigr) = \log 3/4 <0\cr 
            \nu\bigl(\{ 1,2,3 \}\bigr) = \log 4/3.
        \end{cases}
    \end{equation*} 
\end{example}

The next result tells us that with some
additional symmetries, positive association
implies that we are Poisson generated when \( n=3.\)

\begin{theorem}\label{theorem: four equal}
    Consider a probability measure $\mu$ on $\{0,1\}^3$ which is invariant under permutations and interchanging 0 and 1. Then the following are equivalent.
    \begin{enumerate}[label=(\alph*)]
        \item $\mu$ has positive association.\label{item: prop a} 
        \item $\mu$ is an increasing function of i.i.d.\ random variables.\label{item: prop b} 
        \item $\mu$ satisfies the FKG property.\label{item: prop c} 
        \item $\mu$ is in \( \mathcal{R}.\)       \label{item: prop d} 
    \end{enumerate} 
\end{theorem}

\begin{proof}
    The set of measures \( \mu\) as above is just a one parameter family since $p_1\coloneqq P(X \equiv 1) \leq 1/2$ determines the measure given all of the symmetries.  Since~\ref{item: prop c} implies~\ref{item: prop b} implies~\ref{item: prop a} and~\ref{item: prop d}  implies~\ref{item: prop b} implies~\ref{item: prop a} in general, we need only show that~\ref{item: prop a} implies~\ref{item: prop c} and~\ref{item: prop a} implies~\ref{item: prop d}. 

    It is elementary to check that $\mu$ is positively associated if and only if $p_1\ge 1/8$. Under this assumption, it is easy to verify that the FKG property holds (one way to see this is that the model is then the ferromagnetic Curie-Weiss model).

    To see that~\ref{item: prop a} implies~\ref{item: prop d}, one simply checks that the following $\nu$ measure works. Letting $p_2\coloneqq P\bigl(X(1)=X(2)=1,\, X(3)=0\bigr),$ one verifies that \( p_2 = (1-2p_1)/6\) which is then at most $1/8$.
    $\nu$ gives each of the three singletons weight $\log(\frac{p_1+p_2}{p_1})$, each of the three doubletons weight $\log(\frac{p_1}{2(p_1+p_2)^2})$ and the unique three element set weight $\log(\frac{8(p_1+p_2)^3}{p_1})$. One can check that the first and third terms are always non-negative while the second term is non-negative if and only if $p_1\ge 1/8$.
\end{proof}

Poisson representable processes\color{black}{} have a similar flavor to the so-called divide and color model (see~\cite{st})
but they are certainly different.  In the latter model, one takes a random partition of $S$ (with any distribution) and then assigns
all the elements in each partition element 
either 1 or 0 with probability $p$ and $1-p$.
This is done independently for different clusters. We now give some examples
illustrating the difference between these concepts.

\begin{example}[Example 2.17 in~\cite{st}]
    Consider the divide and color process \( X \) corresponding to the two partitions, \( (12,3,4)\) and \( (1,2,34) \) being chosen with equal probability. Letting \( A = \{ X_1=X_2 = 1 \}\) and \( B = \{ X_3 = 
    X_4=1 \}\), one checks that these are increasing but negatively correlated events and hence this does not have positive association. Consequently,  \( X \notin \mathcal{R}\) by Theorem~\ref{theorem:DFK}. 
\end{example}

\begin{example}
    \( X_1,X_2, X_3,X_4\) be i.i.d.\  with \(P(X_1=0)=1/2.\)  For $n \in \{1,2,3\},$ let \( Y_n \coloneqq \max(X_n,X_{n+1}).\) We first leave it to the reader to check that 
    $(Y_1,Y_2,Y_3)$ is Poisson generated by using  
    $$
    \nu \bigl( \{ 1,2 \} \bigr) =
    \nu \bigl( \{ 2,3 \} \bigr) =
    \nu \bigl( \{ 1 \} \bigr) =
    \nu \bigl( \{ 3 \} \bigr) 
    = \log 2.
    $$
    However, we now argue that \( (Y_1,Y_2,Y_3)\) is not a divide and color process.
    To see this, note that \( P (Y_1 = 1) = 3/4\) and hence any divide and color model  would need to be made with \( p = 3/4.\) 
    Assume \( (Y_1,Y_2,Y_3)\) is a divide and color model.
    Since \( Y_1\) and \( Y_3\) are independent, they cannot be in the same partition element. Hence the only 
    possible partitions are \( (1,2,3), (12,3), (1,23)\).
    By symmetry, these are given masses \( p ,\) \( (1-p)/2, \) and \( (1-p)/2\) for some $p$. This implies that 
    \[
        P(Y_1=Y_2=Y_3 = 0 ) = p(1/4)^3 + (1-p)(1/4)^2.
    \]
    Since, by definition, we have \( P(Y_1=Y_2=Y_3= 0) = 1/16,\) it follows that \( p = 0.\)
    Next, note that
    \[
        1/8 = P(Y_1=Y_2=0) = 1/2 (1/4) + 1/2(1/4)^2 = 1/8 + 1/32,
    \]
    a contradiction.
\end{example}

The following lemma states that the set $\mathcal{R}$ is closed in the set of all random
vectors. 
\begin{lemma}\label{lemma: closed}
    Let \( S \)  be countable 
    and let  \( X_n \in \{ 0,1 \}^{S}\) be a sequence of random vectors that converges in distribution 
    to a random vector \( X \in  \{ 0,1 \}^{S}.\)
    If \( X_n \in \mathcal{R}\) for every 
    \( n ,\)  then \( X \in \mathcal{R}.\)
\end{lemma}

\begin{proof}
    Assume first that \( |S|< \infty.\)
    For \( n \geq 1, \) since \( X_n \in \mathcal{R},\)  there is a measure \( \nu_n\) on \( \mathcal{P}(S) \smallsetminus \{ \emptyset \} \) such that \( X_n = X^{\nu_n}.\) Allowing now our measures to take the value $\infty$, we can extract a convergent
    subsequence \( (\nu_{n'})\) of \( (\nu_n)\) converging to some $\nu$
    which is allowed to take the value $\infty$.
    Now \(    X_{n'}=X^{\nu_{n'}}\) converges to both $X$ and to $X^\nu$ and hence 
     $X\in \mathcal{R}$. 
    Applying Lemma~\ref{lemma: finite to infinite}, we obtain the desired conclusion for any  \(S.\) 
\end{proof}

It turns out that domination from below by product measures for translation invariant processes on $\mathbb{Z}^d$ which belong to
\( \mathcal{R}\) has a simple characterization.

\begin{proposition}\label{proposition: domination from below}
    Let \( \nu\) be a translation invariant measure
    on $\mathcal{P}(\mathbb{Z}^d)\smallsetminus \{ \emptyset \}.$ Then \( X^\nu \geq   \Pi_p\) if and only if  \[\nu( \mathcal{S}_A^\cup) \geq -|A| \log (1-p)\] for all boxes \( A\) of the form $\{-n,\ldots,n\}^d$.
\end{proposition}

\begin{proof}
    This follows immediately from Theorem~\ref{theorem:DFK} and~\cite[Theorem 4.1]{ls} using the fact that for any box \( A,\) \( P\bigl(X^\nu(A) \equiv 0 \bigr) = e^{-\nu(\mathcal{S}_A^\cup)}.\)
\end{proof}

\section{Markov and renewal processes}\label{sec: Markov} 

In this section, we begin by proving the following result, which shows that all positively associated Markov chains on \( \{ 0,1 \}^{\mathbb{Z}} \) are in \( \mathcal{R}.\)

To simplify notation in what follows, given a stationary process \( X, \) we define
\[
    c_k \coloneqq P\bigl( X_0=0,\,  X_k=0 \bigr),\quad k \geq 0,
\]
Note that if \( c_0 = P\bigl(X_0=0 \bigr) > 0,\) then by positive association, we have \( c_k \geq c_0^2 > 0 \) for all \( k >0.\)

\begin{theorem}\label{theorem: Markov}
    Let \( X\) be a non-trivial stationary positively associated Markov chain on \( \{ 0,1 \}^{\mathbb{Z}}.\) Then \( X \in \mathcal{R}\) and \( \nu \) is given by 
    \begin{equation}\label{eq: ren nu def Markov}
        \nu(K)=
        \begin{cases}
            \log \frac{c_{|K|-1}c_{|K|+1}}{c_{|K|}^2} &\text{if } K \text{ is a finite interval, and} \cr
            0 &\text{otherwise.}
        \end{cases}
    \end{equation}
\end{theorem} 

\begin{remark}\label{remark: p and r MC}
We note that if \( X \) is a non-constant positively associated \( \{ 0,1 \} \)-valued Markov chain, then its transition matrix can be written as
\begin{equation}\label{eq: pos associated markov chain} 
    P = \begin{pmatrix}
        p_{00}& p_{01} \\ p_{10} & p_{11}
    \end{pmatrix} = \begin{pmatrix}
        1-p(1-r) & p(1-r) \\ pr & 1-pr
    \end{pmatrix},
    \end{equation}
    for some \( p,r \in (0,1).\) Here \( p \) is the probability that the Markov chain rerandomizes (otherwise it stays fixed) and \( r\) is the probability it moves to zero when it rerandomizes. 
    Hence, as easily checked,
    \begin{equation}\label{eq: ck for markov}
        c_k = r (1-p)^k + r^2 \bigl(  1-(1-p)^{k} \bigr), \quad k \geq 0.
    \end{equation}
\end{remark}

The proof of Theorem~\ref{theorem: Markov} will use the following lemma involving renewal processes, which we now define.

\begin{definition}
    Let \( X \) be non-trivial a \( \{0,1\}\)-valued process on \( \mathbb{Z}.\) We say that \( X \) is a renewal process (with respect to 0) if there is a sequence \( (b_n)_{n \geq 1}\) of non-negative real numbers such that \( \sum_{n=1}^\infty b_n \leq 1 \) and for any \( (a_j)_{j \in \mathbb{Z}} \in \{ 0,1 \}^{\mathbb{Z}}\)
    \[
    P\bigl(\min \{ j \geq 1 \colon X_{k+j} = 0 \} = n\mid X_k = 0 , (X_{k-i})_{i=1}^\infty = (a_{k-i})_{i=1}^\infty  \bigr) = b_n,\qquad \text{for all }  k,n \in \mathbb{Z}.
    \]
    We will in addition always assume that \( \sum_{n=1}^\infty b_n = 1\) and that the process is stationary.  
\end{definition}

\begin{lemma}\label{lemma: connected subsets}
    Assume that \( X = X^\nu \in \mathcal{R}\) for some translation invariant measure \( \nu.\) Then \( X \) is a renewal process if and only if \( \nu \) is supported  on finite intervals of \( \mathbb{Z}. \) 
\end{lemma}

\begin{proof}
    We show "only if" direction; the "if" direction is left to the reader.

    Assume that \( X = X^\nu\) is a renewal process. Then, by Theorem~\ref{theorem:DFK}, \( X\) is positively associated, and hence
     \begin{equation}\label{eq: new assumption}
        P(X_{n+1}=0 \mid  X_n = 0)>0\quad \forall  n \in \mathbb{Z}.
    \end{equation}   
    Let \( n \in \mathbb{Z}. \) Then, by definition, we have
    \begin{align*}
         e^{-\nu(\mathcal{S}_{n+1}\smallsetminus 
 \mathcal{S}_n)} = P(X_{n+1}=0 \mid  X_n = 0)> 0.
    \end{align*}
    Since \( X\) is a renewal process, we also have
    \begin{align*}
        P(X_{n+1}=0 \mid  X_n = 0)
        =
        P(X_{n+1}=0 \mid  X_n = 0,\, X_{n-1} = 0,\, \dots)
        =
        e^{-\nu(\mathcal{S}_{n+1}\smallsetminus \bigcup_{j \leq n} \mathcal{S}_j)}.
    \end{align*}
    Combining the two above equations,  we obtain
    \[
        \nu(\mathcal{S}_{n+1}\smallsetminus \mathcal{S}_n) 
        = \nu(\mathcal{S}_{n+1}\smallsetminus \bigcup_{j \leq n} \mathcal{S}_j) < \infty.
    \]
    In particular, this implies that for any \( j \leq n,\) we have 
    \[
    \nu(\mathcal{S}_{n+1}\cap  
 \mathcal{S}_n^c \cap \mathcal{S}_j)
 =
 0.
    \]
    This implies that \( \nu \) is supported on intervals. 
    Assume now for contradiction that \( \nu \) has support on infinite intervals. Then, without loss of generality, we can assume that 
    \[
    \nu \Bigl( \bigl\{ [k,\infty) \cap \mathbb{Z} \colon k \in \mathbb{Z} \bigr\} \Bigr) > 0.
    \]
    Since \( X = X^\nu, \) with strictly positive probability there is \( k \in \mathbb{Z}\) such that \( X_j = 1\) for all \( j \geq k,\) and hence \( X\) cannot be recurrent. 
    This concludes the proof.
\end{proof}

\begin{proof}[Proof of Theorem~\ref{theorem: Markov}]
    
    Assume first that \( X=X^\nu \in \mathcal{R}\)  for some measure \( \nu. \)\color{black}{} 
    Since \(  X\) is stationary, \( \nu\) is translation invariant, and by Lemma~\ref{lemma: connected subsets}, \( \nu\) is supported on finite intervals.
    Letting \( k \geq 1,\) it follows that
    \begin{align*}
        &c_k = P(X_0 = X_k = 0) = e^{-\nu(\mathcal{S}_{\{0,k\}}^\cup)}
        =
        e^{-\nu(\mathcal{S}_0)}e^{-\nu(\mathcal{S}_k)} e^{\nu(\mathcal{S}_{\{0,k\}}^\cap)}
        =
        e^{-\nu(\mathcal{S}_0)}e^{-\nu(\mathcal{S}_0)} e^{\nu(\mathcal{S}_{\{0,k\}}^\cap)}
        \\&\qquad=
        c_0^2 e^{\nu(\mathcal{S}_{\{0,k\}}^\cap)}
        =
        c_0^2 \prod_{j>k} e^{(j-k) \nu([j])}.
    \end{align*}
    From this, it follows that
    \begin{align*}
        &\frac{c_{k-1}c_{k+1}}{c_k^2} = 
        \frac{c_0^2 \prod_{j>k-1} e^{(j-(k-1)) \nu([j])}c_0^2 \prod_{j>k+1} e^{(j-(k+1)) \nu([j])}}{c_0^2 \prod_{j>k} e^{(j-k) \nu([j])}c_0^2 \prod_{j>k} e^{(j-k) \nu([j])}}
        =
        e^{\nu([k])}  .
    \end{align*}
    Consequently, \( \nu \) is given by~\eqref{eq: ren nu def Markov}. 

    Using~\eqref{eq: ck for markov}, it is easy to check that \( X\) being positively associated implies that the following inequality holds.
    \begin{equation}\label{eq: convexity 1 Markov}
        c_{k-1}c_{k+1} \geq c_k^2 \quad \forall k \geq 1.
    \end{equation}  

    Let \( \nu\) be given by~\eqref{eq: ren nu def Markov}.  
    Since \( \nu \) is translation invariant and supported only on finite intervals, it follows from Lemma~\ref{lemma: connected subsets} that \( X^\nu\) is a renewal process.
    Next, note that for any \( k \geq 0,\) we have, for all \( m \geq k+1,\)
    \begin{equation}\label{eq: part of limit}
        \begin{split}
        & \sum_{j=k+1}^m (j-k) \log\frac{c_{j+1}c_{j-1}}{c_j^2} 
        =  \log c_k - (m+1-k) \log c_m  + (m-k)\log c_{m+1}
        \\&\qquad=  \log c_k - \log c_m  + (m-k)\log \frac{c_{m+1}}{c_m}.
        \end{split}
    \end{equation}
    Using~\eqref{eq: ck for markov}, one can verify that
    \begin{equation}\label{eq: thing to be replaced if not Markov}
        \lim_{m \to \infty} c_m = c_0^2 \quad \text{and} \quad \lim_{m \to \infty} m \log \frac{c_{m+1}}{c_m} = 0.
    \end{equation}   
    Letting \( m \to \infty \) in~\eqref{eq: part of limit}, we obtain
    \begin{align*}
        &P(X^\nu_0 = X^\nu_k=0) 
        =
        c_0^2 e^{\sum_{j>k} (j-k) \nu([j])}
        =
        c_0^2 e^{\log c_k - \log c_0^2} = c_k,
    \end{align*}
    where the first equality follows as in the first display of the proof.
    Since \( X^\nu\) and \( X\) are both renewal processes, this shows that \( X^\nu = X.\)
\end{proof}

We now state and prove a more general version of Theorem~\ref{theorem: Markov} which is valid for all renewal processes.

\begin{theorem}\label{theorem: general Markov}
    Let \( X\) be a renewal process with \( P (X_0 = 0) >0. \) Then \( X \in \mathcal{R} \) if and only if
    \begin{equation}\label{eq: convexity 1}
        c_{k-1}c_{k+1} \geq c_k^2 \quad \forall k \geq 1.
    \end{equation}
    Moreover, in this case, \( X = X^\nu,\) where \( \nu \) is the translation invariant measure given by~\eqref{eq: ren nu def Markov}.
\end{theorem}

The proof of Theorem~\ref{theorem: Markov} already proves Theorem~\ref{theorem: general Markov} if we can show that~\eqref{eq: thing to be replaced if not Markov} also holds in this more general setting. This is the purpose of the following lemma.

\begin{lemma}\label{lemma: renewal properties}
    Assume \( X\) is a  renewal process with \( c_0 > 0 \) such that~\eqref{eq: convexity 1} holds.
    Then
    \begin{enumerate}[label=(\alph*)]
        \item \( (c_k)_{k\geq 0}\) is decreasing,\label{item: decreasing}
        \item \( \lim_{k \to \infty} c_k = c_0^2,\) and  \label{item: mixing}
        \item \( \lim_{k \to \infty} k \log  (c_{k+1}/ c_k) =0.\)\label{item: finer convergence}
    \end{enumerate}
\end{lemma}

\begin{proof}\mbox{}
    \begin{enumerate}[label=(\alph*)]
        \item Since~\eqref{eq: convexity 1} holds and \( c_0>0,\) we have \( c_k \in (0,1)\) for all \( k \geq 0.\)
     
    For any \( k \geq 1,\) we have
    \begin{equation*}
        \frac{c_{k+1}c_{k-1}}{c_k^2}\geq 1 \Leftrightarrow
        \frac{c_{k-1}}{ c_k}\geq \frac{c_k}{c_{k+1}} .
    \end{equation*}
    Consequently, the sequence \( (\frac{c_{k-1}}{ c_k})_{k\geq 1}\) is decreasing and converges to a limit \( a \in [0,c_0/c_1]\) as \( k \to \infty. \) 
    We will now show that \( a \geq 1.\) To this end, assume for contradiction that \( a<1.\) Then there is \( j \geq 1\) such that \( c_k/c_{k+1} < (1+a)/2<1\) for all \( k \geq j,\) and hence \( c_{k}> c_j \bigl(2/(1+a)\bigr)^{k-j} \) for all \( k \geq j. \) Since \( c_j >0,\) this implies that \( \lim_{k \to \infty} c_k = \infty,\) contradicting that \( c_k < 1\) for all \( k \geq 0.\) Hence we must have \( a \geq 1.\) Since \( (c_{k-1}/c_{k})_{k \geq 1}\) is decreasing, it follows that \( c_{k-1}/c_{k}\geq a=1\) for all \( k \geq 1, \) and hence \((c_k)_{k\geq 0} \) is decreasing. This completes the proof of~\ref{item: decreasing}.

    \item Since, by~\ref{item: decreasing}, \( (c_k)_{k\geq 0}\) is decreasing, let \( c_\infty\) denote its limit.
    Let 
    \begin{equation*}
        a_k \coloneqq P\bigl(\min\{ j \geq 0 \colon X_j=0 \} = k \bigr),\qquad k\geq 0.
    \end{equation*}
    Then, for any \( k \geq 1,\) we have (due to stationarity)
    \begin{equation*}
        c_0 = P(X_k = 0) = \sum_{j=0}^k a_j (c_{k-j}/c_0).
    \end{equation*}
    Since \( \sum_{j=0}^\infty a_j = 1,\) \( c_0>0,\) and   \( c_k \searrow c_\infty,\) we obtain~\ref{item: mixing}.

    \item 
    To this end, note that for any \( m \geq 1,\) we have
    \begin{align*}
        &S_m \coloneqq \sum_{k=1}^m k \log \frac{c_{k-1}c_{k+1}}{c_{k}^2}
        =
        \log c_0 + m\log c_{m+1} - (m+1)\log c_m
        \\&\qquad=
        \log c_0 - \log c_m + \log (c_{m+1}/ c_m)^m \overset{\rm\ref{item: mixing}}{\leq} \log c_0 - \log c_m + 0 
        \\&\qquad\nearrow \log c_0 - \log c_0^2 = -\log c_0.
    \end{align*} 
    Since each term in the sum \( S_m \) is non-negative, \( S_m\) is increasing in \( m.\) Since \( (S_m)_{m\geq 1} \) is increasing and bounded from above by \( -\log c_0,\) its limit \( \lim_{m \to \infty} S_m\) exists and is bounded from above by \( - \log c_0. \)
    Since, by~\ref{item: mixing}, the limit of \( (c_m)_{m\geq 0}\) also exists, and 
    \begin{equation*}
        S_m -(\log c_0 - \log c_m)
        =
        \log (c_{m+1}/ c_m)^m,
    \end{equation*}
    it follows that  \( \lim_{m \to \infty}  (c_{m+1}/ c_m)^m\) exists.  
    Next, note that since \( c_0 \geq c_1 \geq  \dots \geq c_0^2\) and \( \sum_{m=1}^\infty 1/m = \infty,\) we must have \( \liminf_{m \to \infty} m(c_m-c_{m+1}) = 0.\) Since  
    \begin{align*}
        &\liminf_{m \to \infty} m(c_m-c_{m+1}) = 0
        \Rightarrow  \liminf_{m \to \infty} m(1-c_{m+1}/c_m) = 0 
        \\&\qquad\Rightarrow \liminf_{m \to \infty} (1-(1-c_{m+1}/c_m))^m = 1
         \Leftrightarrow  
        \liminf_{m \to \infty} (c_{m+1}/c_m)^m = 1
    \end{align*}
    and \(  \lim_{m \to \infty}  (c_{m+1}/ c_m)^m\) exists, it follows that 
    \[ 
        \lim_{m \to \infty}  (c_{m+1}/ c_m)^m = \liminf_{m \to \infty}  (c_{m+1}/c_m)^m = 1.
    \]
    This establishes~\ref{item: finer convergence}, and thus completes the proof.
    \end{enumerate}
\end{proof}

\begin{proof}[Proof of Theorem~\ref{theorem: general Markov}]

    Replacing~\eqref{eq: thing to be replaced if not Markov} with Lemma~\ref{lemma: renewal properties}, the proof of Theorem~\ref{theorem: Markov} gives the desired conclusion.
\end{proof}

Lemma~\ref{lemma: connected subsets} tells us that for positively associated Markov chains, the corresponding \( \nu\) is supported on finite intervals. Interestingly, a similar result holds for Markov random fields.

\begin{proposition} \label{prop: MP and connected}
    Let \( X = X^\nu\) be a \(\{0,1\}\)-valued process on a connected graph that satisfies the Markov property and is such that for all finite sets \( A,\)
    \[
        P\bigl(X(A)\equiv 0 \mid  X(\partial A) \equiv  0 \bigr)  > 0.
    \]
    Let \( \mathcal{D}\) be the set of all disconnected subsets of \( S,\) at least one of whose components is finite. Then \( \nu(\mathcal{D})=0.\)  
\end{proposition}

\begin{proof}
    By assumption, we have for any finite set $A$
    \begin{align*}
        P\bigl(X(A)\equiv 0 \mid  X(\partial A) \equiv  0 \bigr) =  e^{-\nu( \mathcal{S}_{A}^\cup \smallsetminus 
  \mathcal{S}_{\partial A}^\cup )} > 0.
    \end{align*}
    Since \( X\) satisfies the Markov property, we also have
    \begin{align*}
        P\bigl(X(A)\equiv 0 \mid  X(\partial A) \equiv  0 \bigr)
        =
        P\bigl(X(A)\equiv 0 \mid  X( A^c) \equiv  0 \bigr)
        = 
        e^{-\nu( \mathcal{S}_A^\cup \smallsetminus \mathcal{S}_{A^c}^\cup)} > 0.
    \end{align*}
    Combining the two above equations, we obtain
    \[
        \nu( \mathcal{S}_A^\cup \smallsetminus 
         \mathcal{S}_{\partial A}^\cup)
        =
        \nu(\mathcal{S}_A^\cup \smallsetminus \mathcal{S}_{A^c}^\cup) <\infty
    \]
    which easily yields 
    \begin{equation*}
        \nu\bigl( (\mathcal{S}_A^\cup \cap \mathcal{S}_{A^c}^\cup) \smallsetminus \mathcal{S}_{\partial A}^\cup \bigr)  =0.
    \end{equation*}
    This concludes the proof. 
\end{proof}

\begin{remark}
    With Theorem~\ref{theorem: Markov} in mind, it is natural to ask whether 
    \begin{enumerate}
        \item all positively associated tree-indexed Markov chains are in \( \mathcal{R}\), and if
        \item the Ising model on \( \mathbb{Z}^d,\) \( d \geq 2,\) is in \( \mathcal{R}.\)
    \end{enumerate}
    We answer both these questions negatively in Theorem~\ref{theorem: MC on trees} and Theorem~\ref{theorem: Ising}.
\end{remark}

\section{Infinite permutation invariant processes}\label{sec: inf perm}

In this section, we consider (possibly infinite)
measures $\nu$ on 
\( \{ 0,1 \}^\mathbb{N}\) which are invariant under finite permutations 
and the associated processes $X^\nu$ which are
permutation invariant probability measures on
\( \{ 0,1 \}^\mathbb{N}\). Also, in
this section, we let \( \Pi_p\) denote the product measure on \( \{ 0,1 \}^S\) with density \( p\)
and we recall de Finetti's Theorem which states that each
$\{0,1\}$-valued process $X$ indexed by $\mathbb{N}$
which is invariant under finite permutations is an average of product measures; i.e.\ there
is a (unique) probability measure $\mu$ on $[0,1]$ such that
the distribution of $X$ is
$$
\int_0^1 \Pi_p \, d\mu(p).
$$
We will identify $X$ with $\mu$.

\begin{example}\label{example: p}
    If \( X \sim \Pi_p,\) i.e. if \( \mu = \delta_p,\)  then \( X = X^\nu\) for the (infinite) measure \( \nu\) defined by
    \begin{equation*}
        \nu(A) = \begin{cases}
            -\log(1-p) &\text{if } A = \{ j \} \text{ for some } j \in \mathbb{N} \cr
            0 &\text{else.}
        \end{cases}
    \end{equation*}
\end{example}

\begin{example}\label{example: prod p1}
    For all \( p \in [0,1], \) if \(  X  \sim \alpha \Pi_p + (1-\alpha)\Pi_1 ,\) i.e. if \( \mu = \alpha\delta_p + (1-\alpha)\delta_1,\)  then  \( X =  X^\nu \) for the (infinite) measure \( \nu\) defined by
    \begin{equation*}
        \nu(A) = 
        \begin{cases}
            -\log (1-p) &\text{if } A=\{ j \} \text{ for some } j \in \mathbb{N} \cr
            -\log \alpha &\text{if } A = \mathbb{N}, \cr
            0 &\text{else.}
        \end{cases}
    \end{equation*}
\end{example}

\begin{example}
    Let \( \nu = \Pi_p. \) Then
    \begin{equation*}
        X^\nu \overset{d}{=} \sum_{j=0}^\infty \frac{e^{-1}}{j!} \Pi_{1-(1-p)^j},
    \end{equation*}
    which is a convex combination of product measures whose densities
    have an accumulation point at one.
\end{example}

\subsection{A de Finetti theorem for infinite measures}
To begin here, we need to first understand the permutation 
invariant (possibly infinite) measures $\nu$
on $\{0,1\}^\mathbb{N}$. If we stick to probability 
measures, then this is the classical
de Finetti's Theorem above.
While this immediately extends to any finite measure,
we need to encompass infinite measures as well here. 
We \color{blue}proved\color{black}{} the following infinite version of de Finetti's Theorem which is relevant for our specific context. After having done this, we learned that this was already done (in a slightly different and even more general setup) by Harry Crane; see~\cite[Theorem 2.7]{crane}. Despite the result therefore not being new, we include our proof below since it is quite short and written using the same language as the rest of the paper.

\begin{theorem}[A version of de Finetti's theorem for possibly infinite measures]\label{theorem: definetti}
    Let \( \nu\) be a permutation invariant measure on \( \mathcal{P}(\mathbb{N})\smallsetminus \{ \emptyset \}.\)
    Assume that \( \nu(\mathcal{S}_{1}) <\infty.\) Then there is a unique measure \( \sigma\) on \( (0,1] \) and \( c \geq 0 \) such that
    \begin{equation}\label{eq: definetti}
        \nu = c \sum_{i \in \mathbb{N}} \delta_i  + \int_0^1 {\textstyle \Pi_x} \,  d\sigma(x)
    \end{equation}
    and
    \begin{equation*}
        \int_0^1 x \, d\sigma(x)<\infty.
    \end{equation*}
\end{theorem}

Note that the assumption that \( \nu(\mathcal{S}_{1}) <\infty\) is needed to not have  \(X^\nu \equiv 1\) \color{black} almost surely. 

\begin{proof}
   Let \( \mathcal{U} \coloneqq \bigl\{ \{ 1 \}, \{ 2 \}, \dots \bigr\}, \) and for \( j \geq 1, \) let \(  \nu_j \coloneqq \nu |_{\mathcal{S}_{[j]}^\cup \smallsetminus (\mathcal{U} \cup \mathcal{S}_{[j-1]}^\cup)}.\) Note that by assumption, for each \( j \geq 1, \) \( \nu_j \) is a finite measure.
    Then we can write
    \begin{equation*}
        \nu = c \sum_{i \in \mathbb{N}} \delta_i + \sum_{j=1}^\infty \nu_j.
    \end{equation*}
    In particular, we have written \( \nu \) as a sum of measures with disjoint supports.

    Define 
    \[
        \tau A  \coloneqq \{ 1 \} \cup (A + 1),\quad A \subseteq \mathbb{N}.
    \]
    and
    \[
        {\hat { \nu}}_1(\cdot ) \coloneqq \nu_1\bigl(\tau \circ \cdot \bigr)
        =
        \nu\bigl(\tau \circ \cdot \bigr).
    \]
    Then \( \hat \nu_1\) is permutation invariant and finite, and hence, by de Finetti's theorem, there is a unique finite measure \( m\) on \( [0,1]\) such that 
    \[
        \hat { \nu}_1 \overset{d}{=} \int \Pi_{x} \, dm(x).
    \]
    Note that
    \[
        \| m \| = \| \hat \nu_1 \| = \nu\bigl(\{ A \colon \min A = 1,\, |A|>1\}\bigr) <\infty.
    \]

    For \( j \geq 1, \) let 
    \[
        \mathcal{A}_j \coloneqq \support  (\nu_j) = \{ A \subseteq \mathbb{N} \colon |A|>1,\, \min A = j \}.
    \]

    For \( \mathcal{S} \subseteq \mathcal{A}_j,\) we have
    \begin{align*}
        &\nu_j(\mathcal{S}   ) = \nu(\mathcal{S}) 
        = 
        \nu(\sigma_{1j} \mathcal{S}  ) 
        = 
        \nu_1(\sigma_{1j}  \mathcal{S})
        = 
        \hat \nu_1\bigl(\tau^{-1}(\sigma_{1j}  \mathcal{S})\bigr)
        \\&\qquad
        \\&\qquad =  
        \int \Pi_{x}\bigl( \tau^{-1}(\sigma_{1j}  \mathcal{S})\bigr) \, dm(x) 
        =  
        \int \Pi_{x}(  \mathcal{S} ) \, x^{-1} \, dm(x) .
    \end{align*}
    Letting \( d\sigma(x) = x^{-1}\, dm(x),\)  the desired conclusion immediately follows.
\end{proof}

\subsection{Characterization in terms of infinite divisibility}\label{section: characterization}

We recall that a random variable \( Z \) on \( [0,\infty)\) is infinitely divisible if its Laplace transform satisfies
\begin{equation}\label{eq: laplace}
    \mathbb{E}[e^{-t Z}] = e^{-\gamma t } e^{\int_0^\infty (e^{-st}-1)\, d\hat \sigma(s)},\quad t\geq 0
\end{equation}
where  \( \gamma \geq 0, \)  \( \smash{\hat \sigma \bigl((\varepsilon,\infty]\bigr) }\) is finite for all \( \varepsilon > 0,\) and \( \int_0^1 s \, d\hat \sigma(s) < \infty.\) Here \( \hat \sigma \) is called the Levy measure of \( Z.\) The notion of an infinitely divisible distribution on \( [0,\infty)\) easily extends to the case where the probability that \( Z = \infty\) is strictly positive. This corresponds to the law  \( p \mathcal{L}(Z') + (1-p) \delta_\infty, \) where \( Z' \) is infinitely divisible. 

Our main theorem completely identifies permutation invariant random vectors in \( \mathcal{R} \) with infinitely divisible distributions on \( [0,\infty].\)

\begin{theorem}\label{theorem: characterization}
\mbox{} 
\begin{enumerate}[label=(\alph*)]
    \item Assume that $ X \in  \mathcal{R}$ satisfies \( X = X^\nu \) where \( \nu\) is as in~\eqref{eq: definetti}. Further, let \( \mu \) be such that \( X \sim \int_0^1 \Pi_p \, d\mu(p).\)  Let \( \varphi \colon x \mapsto -\log(1-x) \) map \([0,1] \) to \( [0,\infty],\) and let \( Z = \varphi(Q) \) where \( Q \sim \mu. \) 
    Then 
    \begin{enumerate}[label=(\roman*)]
        \item $Z$ is infinitely divisible,
        \item $\hat \sigma = \sigma \circ \varphi\mathrlap{^{-1}},\phantom{{}^{1}} $ and  
        \item  \( \gamma = \nu(\{1\}) ,\)  
    \end{enumerate} 
    where \( \hat \sigma\) and \( \gamma\) are as in~\eqref{eq: laplace}.
    \label{item: characterization 1}

    \item Let $Z$ be an infinitely divisible distribution on $[0,\infty].$ Let \( \varphi^{-1} = \phi \colon x \mapsto 1-e^{-x}\) map \( [0,\infty]\) to \( [0,1],\)  and let \( \mu = \mathcal{L}(\phi(Z)).\) Let \( X \sim \int_0^1 \Pi_p \, d\mu(p).\) Then 
    \begin{enumerate}[label=(\roman*)]
        \item $X \in \mathcal{R},$ 
        \item $\sigma = \hat \sigma \circ \phi^{-1},$ and  
        \item \( \nu(\{ 1 \}) = \gamma,\)
    \end{enumerate} 
    where \( \hat \sigma\) and \( \gamma\) are as in~\eqref{eq: laplace} and \( X = X^\nu,\) where \( \nu \) is as in~\eqref{eq: definetti}.
    \label{item: characterization 2}
 \end{enumerate}
\end{theorem}

\begin{lemma}\label{lemma: Malins characterization}
    Let \( X \sim \int_0^1 \Pi_p \, d\mu(p). \) Assume that \( X \in \mathcal{R}\) and let \( \nu = c \sum_{i \in \mathbb{N}} \delta_{\{ i \}} + \int_0^1 \Pi_x \, d\sigma(x)\) be such that \( X = X^\nu.\) 
    Further, let \( Q \sim \mu, \)  \( y_0 \coloneqq 1-e^{-\nu(\{1\})},\) and  \( (Y_j)_{j \geq 1}\) be a Poisson point process with intensity \( \sigma.\) Then
    \begin{equation*}
        Q \overset{d}{=} 1-(1-y_0) \prod_{j\geq 1} (1-Y_j).
    \end{equation*} 
\end{lemma}

\begin{proof}
    Assume without loss of generality that \( Y_1 \geq Y_2 \geq \dots\) 
    Then, one observes that
    \[
        X^\nu \mid (Y_j)_{j\geq 1} \sim \Pi_{1 - (1-y_0)\prod_{j\geq 1} (1-Y_j)}.
    \]
    Since we also have \( X^\nu \sim \int_0^1 \Pi_p \, d\mu(p), \) it follows, since the set of permutation invariant measures is a simplex, that 
    \[
        1-(1-y_0)\prod_{j\geq 1} (1-Y_j) \sim \mu.
    \] 
    This concludes the proof.
\end{proof}

\begin{proof}[Proof of Theorem~\ref{theorem: characterization}]
    Let  \( y_0 \) and  \( (Y_j)_{j \geq 1} \) be as in Lemma~\ref{lemma: Malins characterization}. Then, by this lemma,  we have
    \begin{equation*}
        Q \overset{d}{=} 1-(1-y_0)\prod_{j \geq 1}(1-Y_j) ,
    \end{equation*}
    and hence
    \begin{equation*}
        Z =  -\log (1-Q) \sim -\log (1-y_0) - \sum_{j \geq 1} \log(1-Y_j).
    \end{equation*}
    This implies that for \( t \geq 0\) we have
    \begin{equation}\label{eq: expectation of exp}
        \begin{split}
            & \mathbb{E}[e^{-t Z}] 
            = \mathbb{E}[e^{-t(-\log (1-y_0) - \sum_{j \geq 1} \log(1-Y_j))}]= \mathbb{E}[e^{t \log (1-y_0) +t \sum_{j \geq 1} \log(1-Y_j)}]
            \\&\qquad=
            e^{t \log(1-y_0)}\mathbb{E}[e^{ t \sum_{j \geq 1} \log(1-Y_j)}]
            = e^{t \log(1-y_0)} e^{ \int_0^1 (e^{ t \log(1-y) }-1) \, d  \sigma(y)}
            \\&\qquad =
            e^{t \log(1-y_0)} e^{ \int_0^1 ((1-y)^{t} -1) \, d  \sigma(y)}
            = 
            e^{t \log(1-y_0)} e^{ \int_0^\infty (e^{-st} -1) \, d  (\sigma \circ \varphi^{-1})(s)},
        \end{split}
    \end{equation}
    where we use Campbell's formula in the third to last equality.
    This concludes the proof of~\ref{item: characterization 1}.

    To see that~\ref{item: characterization 2} holds, let \( \nu \) be defined by
    \[
     \nu \coloneqq \sum_{j \in \mathbb{N}} \gamma \delta_{\{ j \}} +  \int_0^1 \Pi_x \, d(\hat \sigma \circ \phi^{-1})(x).\] 
    Let \( X' \coloneqq X^\nu. \) Then, by construction, \( X' \in \mathcal{R}.\) We will now show that \( X \overset{d}{=} X'.\)
    To see this, let \( \mu'\) be the de Finetti measure of  \( X',\)  \( Q' \sim \mu', \) and \( Z' \coloneqq -\log(1-Q').\) Let \( Q \sim \mu \) and \( Z \coloneqq -\log(1-Q).\)
    Since the set of permutation invariant measures is a simplex, it suffices to show that \( \mu = \mu', \) or equivalently, that \( Z \overset{d}{=} Z'.\) 
    Using~\eqref{eq: expectation of exp} with \( Z\) replaced by \( Z',\) we obtain \( \mathbb{E}[e^{-t Z}] = \mathbb{E}[e^{-t Z'}]\) for all \( t \geq 0.\) Hence \( Z \overset{d}{=} Z' .\) This concludes the proof of~\ref{item: characterization 2}.
\end{proof}

The following result is a direct consequence of Theorem~\ref{theorem: characterization}.

\begin{corollary}\label{corollary: finite sum of products}
    Let \( m \geq 2, \) \( 0 \leq p_1 < \dots < p_m \leq 1 \) and \( \alpha_1,\dots, \alpha_m \in (0,1) \) be such that \( \sum_{i=1}^m \alpha_i = 1. \) Let \( X = (X_j)_{j \in \mathbb{N}} \sim \sum_{i=1}^m \alpha_i \Pi_{p_i}.\) Then  \( X \in \mathcal{R}\)  if and only if \( m = 2\) and  \( p_m = 1.\)    
\end{corollary}

\begin{proof}
    We have \( X \sim \int_0^1 \Pi_p \, d\mu(p), \) where \( \mu \) has finite support. Since the only finitely supported infinitely divisible random variables with no support at \( \infty \) are constant, the desired conclusion follows from Theorem~\ref{theorem: characterization}.
\end{proof}

\subsection{Examples}

\begin{example}
    Let \(  \sigma\) be the Lebesgue measure on \( [0,1].\) Further, let \( \nu = \int_0^1 \Pi_x\, d\sigma(x),\) and let \( \mu \) be such that \( X^\nu \sim \int \Pi_p \, d\mu(p).\) Then, by Theorem~\ref{theorem: characterization}\ref{item: characterization 1}, for \(Q \sim \mu \) we have that \( -\log(1-Q)\) is infinitely divisible with \( \gamma=0\) and 
    \[
        d\hat \sigma(x) = e^{-x} \, d\sigma(x) =  e^{-x} \, dx.
    \] 
    On the other hand, assume that \( \mu\) is of the form
    \begin{equation*}
        \mu \sim e^{-1} \delta_0 + (1-e^{-1})\int_0^1 \Pi_p  f(p) \, dp,
    \end{equation*}
    for some "nice" probability density function \( f. \)  
    Let \( (Y_i)_{i\geq 1}\) be a Poisson point process with intensity \( \sigma, \) let \( U_1,U_2,\dots \sim \mathrm{unif}(0,1) \) be i.i.d., and let \( Q \sim \mu. \) Then, using Lemma~\ref{lemma: Malins characterization}, for \( t > 0,\) we have
    \begin{align*}
        &P(Q\leq t) = P\bigl(1-\prod_{i\geq 1} (1-Y_i) \leq t \bigr)
         = P\bigl(1-t \leq \prod_{i\geq 1} (1-Y_i) \bigr)
         \\&\qquad = \sum_{k=0}^\infty \frac{e^{-1} }{k!} P\bigl(1-t \leq \prod_{i= 1}^k (1-U_i) \bigr) = \sum_{k=0}^\infty \frac{e^{-1} }{k!} P\bigl(1-t \leq \prod_{i= 1}^k U_i \bigr)
        \\&\qquad = 1-\sum_{k=1}^\infty \frac{e^{-1} }{k!} P\bigl(\prod_{i= 1}^k U_i < 1-t\bigr)
        = 1-\sum_{k=1}^\infty \frac{e^{-1} }{k!} \int_0^{1-t} \frac{(-\log u)^{k-1}}{(k-1)!} \, du.
    \end{align*}
    Taking derivatives of both sides, we obtain
    \begin{align*}
        &f(t) = \frac{d}{dt}P(Q\leq t)   
        = e^{-1} \sum_{k=1}^\infty \frac{(-\log (1-t))^{k-1}}{k!(k-1)!}  
        =
        e^{-1}
        \frac{J_1\left(2 \sqrt{-\log (1-t)}\right)}{\sqrt{-\log (1-t)}},
    \end{align*}
    where \( J_1\) is the Bessel function of the first kind with order 1. 
\end{example}

\begin{example}\label{ex: 4.9}
    Let \( \mu  \) be the uniform measure on \( [0,1]\) and let \( X \sim \int_0^1 \Pi_p \, d\mu(p).\) Let \( Q \sim \mu\) and let \( Z \coloneqq -\log(1-Q).\) Then 
    \begin{equation*}
        P(Z \leq t) = P(-\log(1-Q) \leq t) 
        = P(Q \leq 1-e^{-t} ) = 1-e^{-t} ,
    \end{equation*}
    and hence \( Z\) has density \( f(t) = e^{-t}.\) Since the exponential distribution is infinitely divisible, it follows from Theorem~\ref{theorem: characterization}\ref{item: characterization 2} that \( X \in \mathcal{R}.\)  
    Note that it is well known and easily checked that the exponential distribution with parameter 1 corresponds to \( \gamma=0\) and  \( d\hat \sigma(t) = e^{-t}/t\, dt\) in~\eqref{eq: laplace}. Again using Theorem~\ref{theorem: characterization}\ref{item: characterization 2}, we further obtain in~\eqref{eq: definetti},
    \begin{equation}
        d \sigma(x) = e^{-(-\log(1-x))}/(-\log(1-x)) \cdot \frac{1}{1-x}\, dx = \frac{1}{-\log(1-x)} \, dx
    \end{equation} 
    and
    \[
    \nu(\{1\}) = 0.
    \]
\end{example}

\begin{example}\label{ex: 4.9 beta}
    Let \( \mu\) be a Beta-distribution with parameters \( \alpha , \beta \in (0,\infty)\)  (i.e. let \( \mu\) have density \( x^{\alpha-1} (1-x)^{\beta-1} B(\alpha,\beta)^{-1}\)) and let \( Q \sim \mu.\) Then \( - \log (1-Q)\) is infinitely divisible (see e.g.~\cite[Example VI.12.21]{sh}) with \( \gamma = 0\) and Levy measure
    \[
    d\hat \sigma(t) = \frac{e^{-\beta t} (1-e^{-\alpha t})}{t (1-e^{-t})}\, dt.
    \]
    Hence, by Theorem~\ref{theorem: characterization}\ref{item: characterization 2}, \( X \sim \int_0^1 \Pi_p \, d\mu(p)\) is in \( \mathcal{R}.\) Moreover, we get
    \[
    d \sigma(x)
    = 
    d\hat \sigma(-\log(1-x)) \cdot \frac{1}{1-x}
    = 
    \frac{ (1-x)^{\beta-1} (1-(1-x)^\alpha)}{-x \log(1-x)} \, dx
    \]
    and \( \nu\bigl(\{ 1 \}\bigr) = 0.\) We recover Example~\ref{ex: 4.9} by letting \( \alpha=\beta=1.\)
\end{example}

\section{Finite permutation invariant processes}
  
\subsection{The Curie Weiss model and a finite permutation invariant result}\label{sec: curie weiss}

In this section, we state and prove a theorem in the finite permutation invariant setting, which we then use to show that the supercritical Curie Weiss model is not in \( \mathcal{R} \) for large~\( n.\)

\begin{theorem}\label{theorem: variance and limit}
    For each \( n \geq 1, \) let \( X_n \in \{ 0,1 \}^n\) be permutation invariant. Let \( \bar X_n \coloneqq \bigl(X_n(1)+\dots X_n(n)\bigr)/n. \) Assume that there is \( c \in (0,1),\) such that 
    \begin{equation}\label{eq: limit assumption}
        \lim_{n\to \infty} P(\bar X_n \geq c) = 0
    \end{equation}
    and that \( X_n \in \mathcal{R} \) for each \( n \geq 1. \)  
    Then
    \begin{equation}\label{eq: limit assumption ii}
        \lim_{n \to \infty} \Var(\bar X_n) = 0.
    \end{equation}
\end{theorem}

\begin{proof}
    Let \( \delta \in (0,1), \) and let \( \mathcal{S}^\delta\) be the set of all subsets of \( [n]\) with at least \( \delta n\) elements.

    Let \( \nu_n \geq 0\) be such that \( X = X^{\nu_n}\) (such a \( \nu_n\) exists since \( X_n \in \mathcal{R}\)).

    We will first show that \( \lim_{n \to \infty} \nu_n(\mathcal{S}^\delta) =0.\) To this end, let \( c \in (0,1) \) be such that~\eqref{eq: limit assumption} holds and choose $k$ such that $(1-\delta)^k < \frac{1-c}{2}$.  We now condition on the event $\mathcal{E}_{k,n}$ that the Poisson random variable corresponding to the number of sets chosen from ${S}^\delta$ for the $n$th system is at least $k$. Note that
    \[
        P(\mathcal{E}_{k,n})= \sum_{\ell = k}^\infty
            \frac{e^{-\nu_n(\mathcal{S}^\delta)} \nu_n(\mathcal{S}^\delta)^\ell}{\ell!}.
    \]
    One easily sees that for each $i\in [n]$,
    \begin{equation*}
        P\bigl(X_n(i)= 1\mid \mathcal{E}_{k,n} \bigr)\ge1- (1-\delta)^k \ge \frac{1+c}{2}.
    \end{equation*}
   Applying Markov's inequality conditioned on $\mathcal{E}_{k,n}$, we get
    \begin{align*}
        &P\bigl(\bar X_n \ge c\mid \mathcal{E}_{k,n}\bigr) 
        =
        1-P\Bigl(\sum_{i=1}^n \bigl(1-X_n(i)\bigr)    > (1-c)n\mid \mathcal{E}_{k,n}\Bigr)
        \\&\qquad\ge 1-\frac{\mathbb{E}\bigl[ \sum_{i=1}^n \bigl(1-X_n(i)\bigr)\mid \mathcal{E}_{k,n} \bigr] }  {(1-c)n}
        \ge 1-\frac{(1-c)n/2}{(1-c)n}=1/2.
    \end{align*}
    From this, it follows that
    \begin{equation*}        
        P(\bar X_n \ge c)\ge \frac{P(\mathcal{E}_{k,n})}{2}.
    \end{equation*}
    Using~\eqref{eq: limit assumption}, we obtain
    \[
    \lim_{n \to \infty} P(\mathcal{E}_{k,n}) = 0
    \]
    which implies that
    $\lim_{n\to\infty}\nu_n(\mathcal{S}^\delta)=0$. 
    %

    Now note that, by symmetry, we can write
    \begin{align*}
    	\nu_n(\mathcal{S}_1) = \nu_n(\{ 1\}) + \sum_{j=2}^n \nu_n([j]) \binom{n-1}{j-1} \quad \text{and} \quad \nu_n(\mathcal{S}_{[2]}^\cap ) = \sum_{j=2}^n \nu_n([j]) \binom{n-2}{j-2},
    \end{align*}
    and hence, 
    \begin{align*}
    	&\nu_n(\mathcal{S}_1) \geq  \sum_{j=2}^n \nu_n([j]) \binom{n-1}{j-1}
    	=
    	 \sum_{j=2}^n \nu_n([j]) \binom{n-2}{j-2}\frac{n-1}{j-1}
    	\geq 
    	 \sum_{j=2}^{\delta n} \nu_n([j]) \binom{n-2}{j-2}\frac{n-1}{j-1}
    	\\&\qquad \geq 
    	\frac{n-1}{\delta n} \sum_{j=2}^{\delta n} \nu_n([j]) \binom{n-2}{j-2}
    	\geq \frac{n-1}{\delta n}   \nu_n(\mathcal{S}_{[2]}^\cap \smallsetminus \mathcal{S}^\delta).
    \end{align*}
    Since  $\lim_{n\to\infty}\nu_n(\mathcal{S}^\delta)=0$ for any \( \delta \in (0,1) \) and \( \limsup_{n \to \infty} \nu_n(\mathcal{S}_1) < \infty \) by~\eqref{eq: limit assumption}, it follows that $\lim_{n\to\infty}\nu_n(\mathcal{S}_{[2]}^\cap )=0.$
    \color{black}
    Now note that 
    \begin{align*}
        \mathbb{E}[\bar X_n ] =  \mathbb{E}\bigl[X_n(1)\bigr] = 1-e^{-\nu_n(\mathcal{S}_1)}
    \end{align*} 
    and that
    \begin{align*}
        &\mathbb{E}[\bar X_n^2 ] =  n^{-2}\mathbb{E}\Bigl[\bigl(\sum_{i=1}^n X_n(i)\bigr)^2\Bigr] 
        =  
        n^{-2}\biggl( n \mathbb{E}[X_n(1)^2] + n(n-1) \mathbb{E}\bigl[X_n(1) X_n(2)\bigr] \biggr)
        \\&\qquad=  
        n^{-2}\bigl( n \mathbb{E}\bigl[X_n(1)\bigr] + n(n-1) \mathbb{E}\bigl[X_n(1) X_n(2)\bigr] \bigr)
        \\&\qquad=
        n^{-2}\biggl( n \mathbb{E}[X_n(1)] + n(n-1) \Bigl( \bigl( 1-e^{-\nu_n(\mathcal{S}_{[2]}^\cap)}\bigr) +e^{-\nu_n(\mathcal{S}_{[2]}^\cap)} \bigl(1-e^{-(\nu_n(\mathcal{S}_1) -\nu_n( \mathcal{S}_{[2]}^\cap))} \bigr)^2\Bigr)  \biggr)
        \\&\qquad=
        n^{-1} \mathbb{E}[X_n(1)] + (1-n^{-1})  \bigl( 1  -2   e^{-\nu_n(\mathcal{S}_1  )} + e^{-2\nu_n(\mathcal{S}_1) +\nu_n( \mathcal{S}_{[2]}^\cap)} \bigr)  
        \\&\qquad=
        n^{-1} \mathbb{E}[X_n(1)] + (1-n^{-1})  \bigl( (1  -   e^{-\nu_n(\mathcal{S}_1  )})^2 - e^{-2\nu_n(\mathcal{S}_1  )} (1- e^{ \nu_n( \mathcal{S}_{[2]}^\cap)}) \bigr)   .
    \end{align*}
    Combining these equations, we obtain
    \begin{align*} 
    	\Var (\bar X_n )
    	&=
    	n^{-1} \mathbb{E}[X_n(1)] - n^{-1}  \bigl( (1  -   e^{-\nu_n(\mathcal{S}_1  )})^2 - e^{-2\nu_n(\mathcal{S}_1  )} (1- e^{ \nu_n( \mathcal{S}_{[2]}^\cap)}) \bigr) 
    	 \\&\qquad\qquad  - e^{-2\nu_n(\mathcal{S}_1  )} (1- e^{ \nu_n( \mathcal{S}_{[2]}^\cap)})  .
    \end{align*} 
    Since, \( \lim_{n \to \infty} \nu_n( \mathcal{S}_{[2]}^\cap) = 0, \) the desired conclusion immediately follows.
\end{proof}

 Throughout this section, we let \( \mu_{J,n}\) be the Curie-Weiss model on \( [n] \) with parameter \( J=\beta/n,\) so that
    \begin{align*}
        &\mu_{J,n}(\sigma) = Z_{J,n}^{-1} e^{J \sum_{i < j} \sigma_i \sigma_j },\quad \sigma \in \{ -1,1 \}^n .
    \end{align*}
    In other words, we let \( \mu_{J,n} \) be the Ising model with inverse temperature \( J \) on the complete graph with vertices labeled by \( [n] .\) 
    It is well known that there is \( \beta_c \in (0,\infty )\) and  \( (c_\beta) \) such that  if \( \sigma  \sim \mu_{n,J} \) then \( \bar \sigma \) concentrates at \( 0 \pm c_\beta\) as \( n \to \infty,\) where
    \begin{enumerate}
        \item \( c_\beta=0\) if \( \beta \leq \beta_c.\) 
        
        \item \( 0 < c_\beta < 1 \) if \( \beta > \beta_c.\)   
    \end{enumerate}
    \color{black}

\begin{theorem}\label{theorem: curie weiss}
    Let  \( \mu_{J,n}\) be the Curie-Weiss model on \( [n] \) with parameter \( J=\beta/n,\) 
    and let \( X_n \in \{ 0,1\}^n\) be the corresponding \( \{0,1\}\)-valued random vector
    where we have identified $-1$ with $0$. 
    If \( \beta > \beta_c\)  and \( n \) is sufficiently large, then the Curie-Weiss model is not in $\mathcal{R}$.
\end{theorem}

\begin{proof}
    Applying Theorem~\ref{theorem: variance and limit}, the desired conclusion immediately follows.
\end{proof}

\subsection{Finite averages of product measures}

    In Corollary~\ref{corollary: finite sum of products}, we considered finite averages of product measures, and showed that these were not in \( \mathcal{R}\) except in very few special cases.
    In this setting, given \( X \notin \mathcal{R}\), there is an $N\ge 3$ such
    that \( X([n]) \in \mathcal{R}\) for $n<N$ and  \( X([n]) \notin \mathcal{R}\) for 
    $n\ge N$.
    The reason for the "3" is that any average of product measures $X$  is positively associated and hence by the discussion after Lemma~\ref{lemma: finite to infinite}, $X([2])$ is always in $\mathcal{R}$.

    Our averages of product measures have precisely $2m-1$ parameters. The next theorem says that we can reduce it to $2m-2$ parameters. This might not seem like a giant improvement, but when $m=2$, we then have only two parameters. This is the case which we will  analyze in most detail and we can then visualize the phase diagram reasonably well since it is just two-dimensional.
    In this theorem, we use the following notation. Given \( \mathbf{x} = (x_1,x_2,\dots, x_m) \) such that  \( 1= x_1 > x_2 > \dots > x_m \geq 0, \) \( q \in (0,1], \) and \( \pmb{\alpha} = (\alpha_1,\dots , \alpha_m) \in (0,1)^m\) such that \( \sum_{i=1}^m \alpha_i = 1, \) we let 
\[
    X^{q,\mathbf{x},\pmb{\alpha}} \sim \sum_{i=1}^m \alpha_i \Pi_{1-qx_i}.
\] 
Using this particular form for the product measures is advantageous since, as stated in the following theorem, being representable turns
out to be independent of \( q.\)

\begin{theorem}\label{proposition: stright lines}
    Let \( n \geq 2. \) Let \( \mathbf{x} = (x_1,x_2,\dots, x_m) \) be such that  \( 1= x_1 > x_2 > \dots > x_m \geq 0, \) \( q \in (0,1], \) and \( \pmb{\alpha} = (\alpha_1,\dots , \alpha_m) \in (0,1)^m\) be such that \( \sum_{i=1}^m \alpha_i = 1. \) 
    Then
    \( X^{q,\mathbf{x},\pmb{\alpha}}([n])  \in \mathcal{R}\) if and only if  \( X^{q'
    ,\mathbf{x},\pmb{\alpha}}([n]) \in \mathcal{R}\) for all \( q' \in (0,1]. \)
\end{theorem}

The proof of Theorem~\ref{proposition: stright lines} will use the following lemma.

\begin{lemma}\label{lemma: form of nu}
    Let \( n \geq 2.\) 
    Let \( \mathbf{x} = (x_1,x_2,\dots, x_m) \) be such that  \( 1= x_1 > x_2 > \dots > x_m \geq 0, \) \( q \in (0,1], \) and \( \pmb{\alpha} = (\alpha_1,\dots , \alpha_m) \in (0,1)^m\) be such that \( \sum_{i=1}^m \alpha_i = 1. \) 
    Let \( \nu_n\) be the unique signed measure corresponding to \( X^{q,\mathbf{x},\pmb{\alpha}}([n]).\) Then, for \( \ell \in [n],\)
    \begin{align}   
        &\nu_n\bigl([\ell]\bigr)
        = \sum_{j=0}^\ell  (-1)^{\ell-j} 
        \binom{\ell}{j}  \log  \sum_{i=1}^m \alpha_i (qx_i)^{n-j}  \label{eq: general probability}
        \\&\qquad=\label{eq: general probability ii}
        \begin{cases}
        -\log q + \log \frac{\sum_{i=1}^m \alpha_i  x_i^{{n-1}} 
        }{ \sum_{i=1}^m \alpha_i   x_i^{n}  }    &\text{if } \ell =1\cr 
        \sum_{j=0}^\ell  (-1)^{\ell-j} 
        \binom{\ell}{j}  \log  \sum_{i=1}^m \alpha_i x_i^{n-j} &\text{if } \ell \in \{ 2,3, \dots, n \}.
        \end{cases} 
    \end{align}
    Moreover, \( \nu_n(\{1\}) \geq -\log q\) and \( \nu_n([2])\geq 0.\)
\end{lemma}

\begin{proof}
    For \( j \in [n],\) we have \[
    P\bigl(X\bigl([n]\smallsetminus [j]\bigr) \equiv 0 \bigr) = \sum_{i=1}^m \alpha_i (qx_i)^{n-j}.
    \]
    Using~\eqref{eq: Mobius inversion result},  we obtain~\eqref{eq: general probability}. From this, noting that 
    \(
    \sum_{j=0}^\ell  (-1)^{\ell-j} 
        \binom{\ell}{j}  (n-j)\log   q = 0
    \)
    if \( \ell \geq 2,\) we obtain~\eqref{eq: general probability ii}.
    From this, using the Cauchy-Schwarz inequality for the second inequality, we obtain
    \begin{align*} 
        &\nu_n \bigl(\{1\} \bigr)
        =   
        -\log q + \log \frac{\sum_{i=1}^m \alpha_i  x_i^{{n-1}} 
        }{ \sum_{i=1}^m \alpha_i   x_i^{n}  }   \geq - \log q,
    \end{align*} 
    and
    \begin{align*} 
        &\nu_n(\{ 1,2 \})
        =   
        \log \frac{\sum_{i=1}^m \alpha_i   x_i^{n-2}    \sum_{i=1}^m \alpha_i   x_i^n}{(\sum_{i=1}^m \alpha_i   x_i^{n-1} )^2} \geq 0.
    \end{align*}
\end{proof}

\begin{proof}[Proof of Theorem~\ref{proposition: stright lines}]
    Assume that \( X^{q,\mathbf{x},\pmb{\alpha}}([n])  \in \mathcal{R}.\) Then there is a measure \( \nu\) on \( \mathcal{P}([n])\smallsetminus \{ \emptyset \} \) such that \( X^{q,\mathbf{x},\pmb{\alpha}}([n]) = X^\nu.\) 
    Let \( \nu'\) be the unique signed measure corresponding to \( X^{q',\mathbf{x},\pmb{\alpha}}.\) 
    Let \( A \subseteq [n] \) satisfy \( |A|>1.\)
    By Lemma~\ref{lemma: form of nu}, we then have that \( \nu(A)= \nu'(A),\) and hence, since \( \nu \) is a non-negative measure it follows that \( \nu'(A)\geq 0 .\)  
    Finally, we note that, again by Lemma~\ref{lemma: form of nu}, we have \( \nu'(\{1\}) \geq - \log q' \geq 0.\) Hence \( \nu'\) is a non-negative measure, implying that \( X^{q',\mathbf{x},\pmb{\alpha}}([n]) \in \mathcal{R}.\) This concludes the proof.
\end{proof}

The following proposition begins to give information about the phase diagram for when \( X([n]) \in \mathcal{R} \) by explaining what happens when $\alpha_1$ is sufficiently close to zero or one \color{blue}or\color{black}{} when $x_2$ is sufficiently close to zero (see Figure~\ref{fig: finite case rays 3}). (Recall that by Theorem~\ref{proposition: stright lines}, for any \( n, \) \( X([n]) \) being in \( \mathcal{R}\) is independent of~\( q.\))

\begin{proposition}
    Let \( \mathbf{x} = (x_1,x_2,\dots, x_m) \) be such that  \( 1= x_1 > x_2 > \dots > x_m \geq 0, \) \( q \in (0,1], \) and \( \pmb{\alpha} = (\alpha_1,\dots , \alpha_m) \in (0,1)^m\) be such that \( \sum_{i=1}^m \alpha_i = 1. \) Let \( X \sim \sum_{i=1}^m \alpha_i \Pi_{1-qx_i}\) and \( n \geq 3.\) 
    \begin{enumerate}[label=(\alph*)]
        
        \item If \( m \geq 2, \) then   \( X([n]) \in \mathcal{R}\) if \( x_2 \) is sufficiently close to \(0\) (depending on \( \alpha_1 \))\color{black}.\label{item: asymptotics 3}

        \item If \( m \geq 2, \) then   \( X([n]) \in \mathcal{R}\) if  \( \alpha_1 \) is sufficiently close to \(1\).\label{item: asymptotics 6}

        \item If  \( m = 2, \) then   \( X([n]) \notin \mathcal{R}\) if \( x_2>0\) and \( \alpha_1 \) is sufficiently close to \(0\) (depending on \( x_2 \) and \( n\))\color{black}. (This is not true in general for \( m \geq 3.\))\label{item: asymptotics 7}
    \end{enumerate}    
\end{proposition}

\begin{proof}
     Assume first that \( x_2 = 0.\) In this case, one verifies that
    \begin{align*}
        &\nu_n\bigl([k]\bigr) = \begin{cases}
            -\log q &\text{if } k =1 \cr 
            -\log \alpha_1&\text{if } k =n \cr 
            0&\text{otherwise,}
        \end{cases} 
        \end{align*}
    and hence \( X([n]) \in \mathcal{R}.\)
    In the rest of the proof, we therefore assume that \( x_2>0.\)
    
    We first show that~\ref{item: asymptotics 3} holds. 
    To this end, assume that  \( x_2\) is small enough so that \( \alpha_1^{-1}\sum_{i=2}^m \alpha_i x_i <1. \) Then, using~\eqref{eq: general probability ii}, a Taylor expansion, and the assumption that \( \sum_{i=1}^m \alpha_i = 1, \) we obtain that, for any  \( k \in \{ 2,3, \dots , n \}, \)  
     \begin{align*}
        &\nu_n\bigl([k]\bigr)
        = \sum_{j=0}^{\min(k,n-1)}  (-1)^{k-j} 
        \binom{k}{j}  \log \bigl(\alpha_1 +  \sum_{i=2}^m \alpha_i x_i^{n-j}\bigr) 
        \\&\qquad=
        \sum_{j=0}^{\min(k,n-1)}  (-1)^{k-j} 
        \binom{k}{j} \biggl( \log \alpha_1 + \sum_{\ell=1}^\infty \frac{(-1)^{\ell+1}}{\ell}  \bigl(\alpha_1^{-1}\sum_{i=2}^m \alpha_i x_i^{n-j}\bigr)^\ell  \biggr)
        \\&\qquad=
        \begin{cases}
            \alpha_1^{-1}\sum_{i=2}^m \alpha_i x_i^{n-k} +O(x_2^{n-k+1}) &\text{if } k \in \{ 2,3,\dots, n-1 \} \cr  
            -\log \alpha_1 + O(x_2)&\text{if } k=n.
        \end{cases}
    \end{align*} 
    Finally, we note that by Lemma~\ref{lemma: form of nu}, \( \nu_n(\{1\})\geq -\log q.\)  From this~\ref{item: asymptotics 3} immediately follows.

    We now show that~\ref{item: asymptotics 6} holds. To this end, assume that \( \alpha_1\) is close enough to \(1\) to ensure that \( \alpha_1^{-1}\sum_{i=2}^m \alpha_i x_i <1. \) Then, for \( k \in \{ 2,3, \dots, n\},\) as above, it follows that for \(\alpha_1 \to 1\), 
    \begin{align*}
        &\nu_n\bigl([k]\bigr)
        =  
        \sum_{j=0}^{\min(k,n-1)}  (-1)^{k-j} 
        \binom{k}{j} \biggl( \log \alpha_1 + \sum_{\ell=1}^\infty \frac{(-1)^{\ell+1}}{\ell}  \bigl(\alpha_1^{-1}\sum_{i=2}^m \alpha_i x_i^{n-j}\bigr)^\ell  \biggr)
        \\&\qquad= 
        \alpha_1^{-1} \sum_{i=2}^m \alpha_i x_i^{n-k} (1-x_i)^k   +O((1-\alpha_1)^2) . 
    \end{align*} 
    The last equality, when \( k=n,\) uses the observation that \( \log \alpha_1 + \alpha_1^{-1}(1-\alpha_1) = O((1-\alpha_1)^2) .\) 
    Since \( \nu_n(\{1\})\geq -\log q\) by Lemma~\ref{lemma: form of nu}, it follows that \( X([n]) \in \mathcal{R}\) when \( \alpha_1\) is close to \(1\). This concludes the proof of~\ref{item: asymptotics 6}.

    We now show that~\ref{item: asymptotics 7} holds. To this end, assume that \( m = 2,\) and that  \( \alpha_1\) is close enough to \(0\) to ensure that \( \alpha_1    <(1-\alpha_1) x_2^{n}. \) Then, for \( k \in \{ 2,3, \dots, n\},\) by~\eqref{eq: general probability ii} and a Taylor expansion, it follows that for \(\alpha_1 \to 0\),
    \begin{align*}
        &\nu_n\bigl([k]\bigr)
        = 
        \sum_{j=0}^k  (-1)^{k-j} 
        \binom{k}{j}  \log \bigl(\alpha_1 +   \alpha_2 x_2^{n-j}\bigr) 
        \\&\qquad=
        \sum_{j=0}^k  (-1)^{k-j} 
        \binom{k}{j} \biggl( \log (   \alpha_2 x_2^{n-j} )+ \sum_{\ell=1}^\infty \frac{(-1)^{\ell+1}}{\ell} \bigl(\alpha_1 /   (\alpha_2 x_2^{n-j})\bigr)^\ell \biggr)
        \\&\qquad=
        \sum_{j=0}^k  (-1)^{k-j} 
        \binom{k}{j} \bigl( \log  (\alpha_2 x_2^{n-j})+  \alpha_1 /  (\alpha_2 x_2^{n-j} ) \bigr)+ O(\alpha_1^2).
    \end{align*}
    Noting that
    \[
    \sum_{j=0}^k  (-1)^{k-j} 
        \binom{k}{j}   \log  \alpha_2  
        = 
        0
        =
    \sum_{j=0}^k  (-1)^{k-j} 
        \binom{k}{j} (n-j)\log  x_2 ,
    \]
    it follows that
    \begin{align*} 
        &\nu_n\bigl([k]\bigr)
        = 
        (-1)^k \frac{\alpha_1 }{1-\alpha_1} x_2^{-n}(1-x_2)^k  + O(\alpha_1^2).
    \end{align*} 
    This concludes the proof of~\ref{item: asymptotics 7}. 
\end{proof}

The case where $m=2$ and $x_2$ is taken to be close to  \(1\) (corresponding to two product measures having similar densities) is the most intriguing and undergoes a phase transition in $\alpha_1$ where the critical value corresponds
to the largest negative zero of the so-called polylogarithm function as we will see in the following theorem.
In addition, this result will later be used to show that certain positively associated tree-indexed Markov chains, as well as the Ising model with certain parameters, are not in \( \mathcal{R}.\)

\begin{figure}[htp]
    \centering
    
    \begin{subfigure}[b]{0.3\textwidth}\centering
    \includegraphics[width=\textwidth]{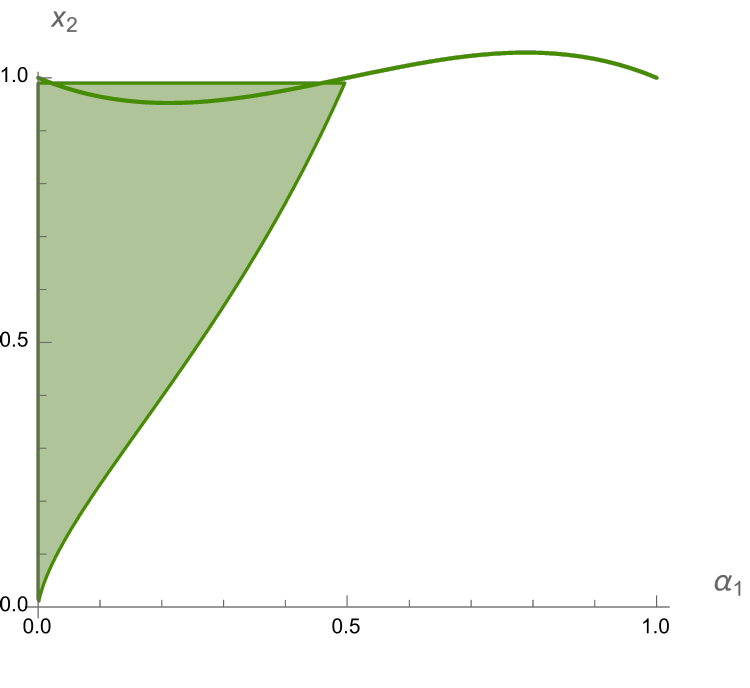}
    \caption{\( n=3,\) \( k=3\)}
    \end{subfigure}
    
    \begin{subfigure}[b]{0.3\textwidth}\centering
    \includegraphics[width=\textwidth]{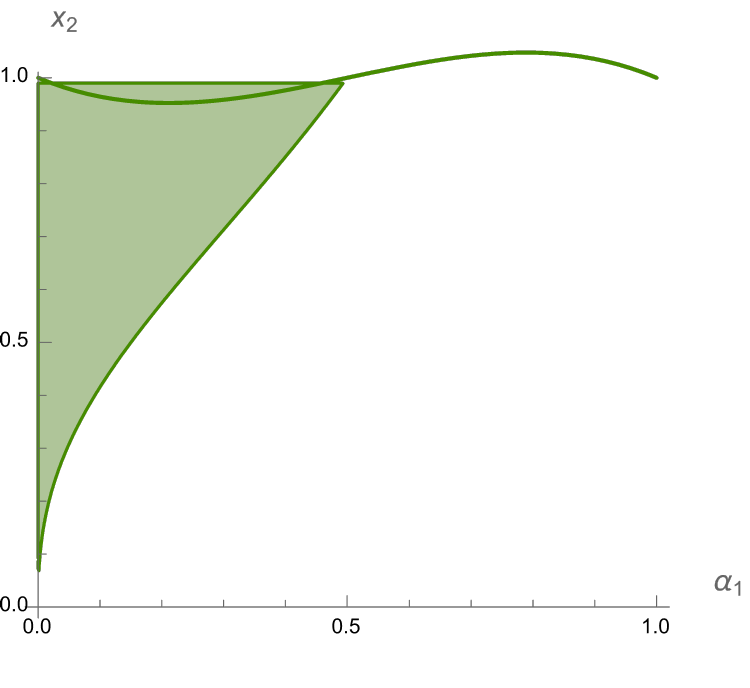}
    \caption{\( n=4,\) \( k =3\)}
    \end{subfigure}
    \begin{subfigure}[b]{0.3\textwidth}\centering
    \includegraphics[width=\textwidth]{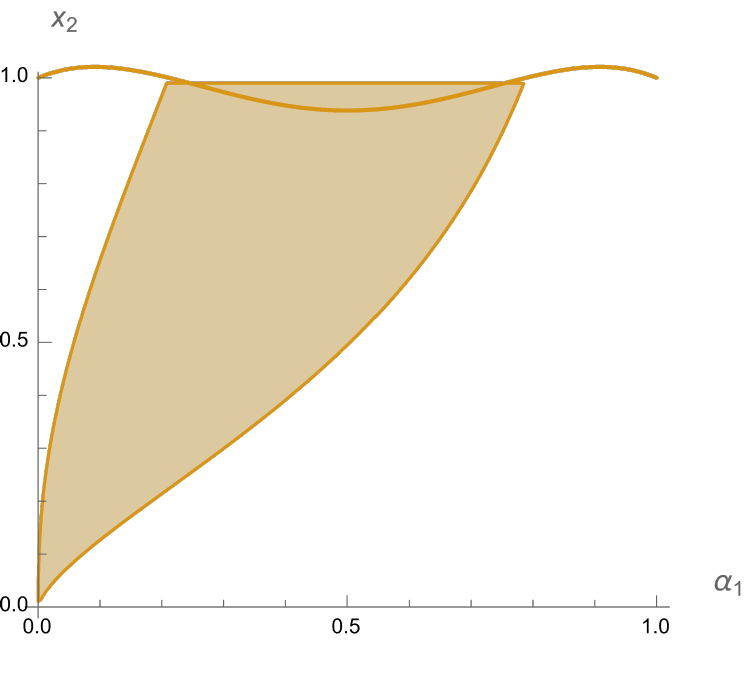}
    \caption{\( n=4,\) \( k=4\)}
    \end{subfigure}

    \begin{subfigure}[b]{0.3\textwidth}\centering
    \includegraphics[width=\textwidth]{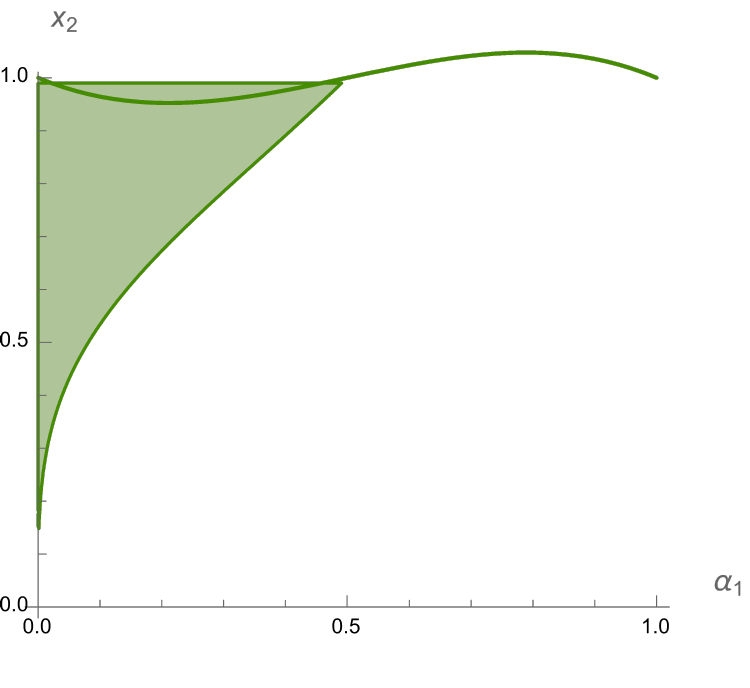}
    \caption{\( n=5,\) \( k=3\)}
    \end{subfigure}
    \begin{subfigure}[b]{0.3\textwidth}\centering
    \includegraphics[width=\textwidth]{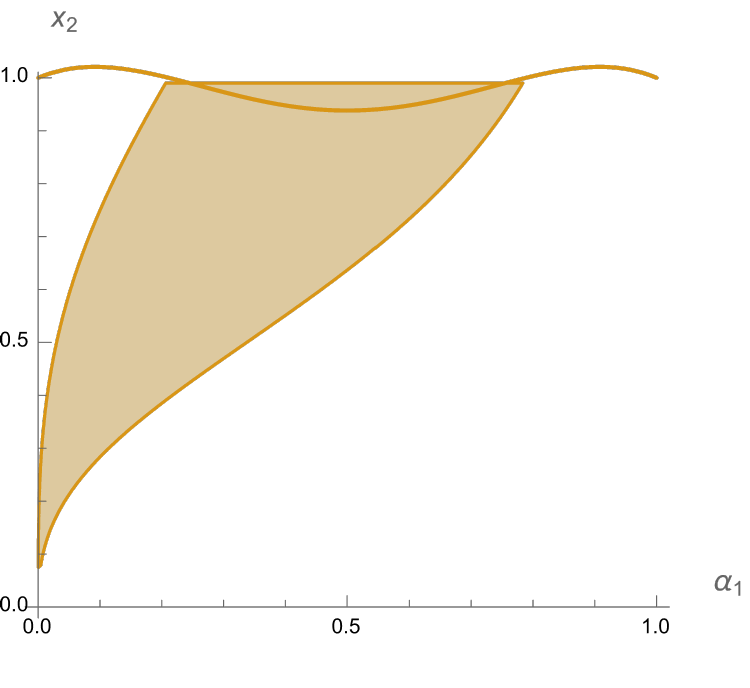}
    \caption{\( n=5,\) \( k =4\)}
    \end{subfigure}
    \begin{subfigure}[b]{0.3\textwidth}\centering
    \includegraphics[width=\textwidth]{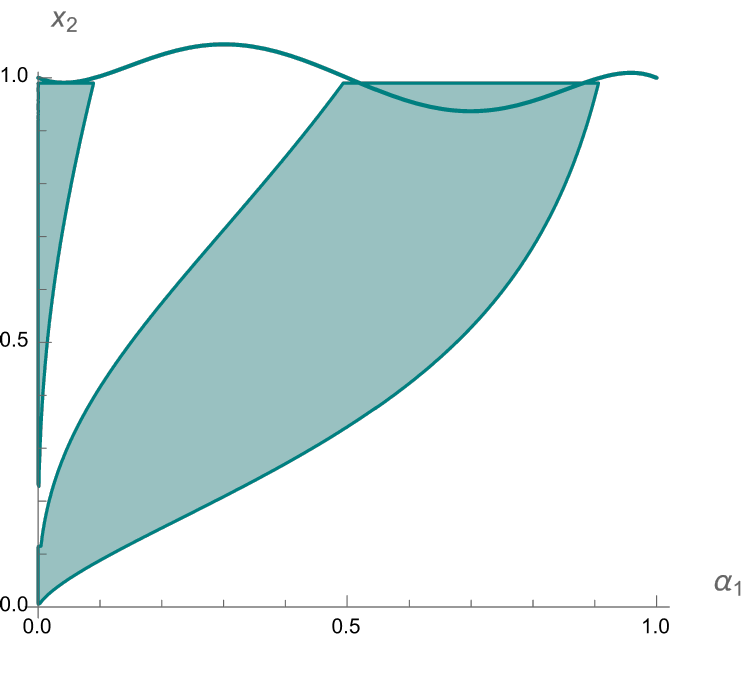}
    \caption{\( n=5,\) \( k =5\)}
    \end{subfigure}

    \caption{In the figures above, we draw the regions where \( \nu_n([k])<0 \) as a function of (our only parameters) \(x_2 \) (on the \( y\)-axis) and \( \alpha_1 \) (on the \( x\)-axis) for \( m=2,\) \( n=3,4,5,\) and \( k\in \{3,\dots, n\}.\) The curves along the \( x\)-axis are the polylogarithm functions which appear in Theorem~\ref{proposition: finite cases}\ref{item: general 1}.}
    \label{fig: finite case rays 3}
\end{figure}

In this result, we will use \( \Li_s \) to denote the polylogarithm function with index \( s \in \mathbb{R},\) defined for \( z \in \mathbb{C}\) with \( |z|<1\) by
\begin{equation*}
    \Li_s(z) \coloneqq \sum_{k=1}^\infty \frac{z^k}{k^s}
\end{equation*}
and extended to  all \( z \in \mathbb{C}\) (except for poles) by analytic continuation. When \( s\) is a negative integer the function \(\Li_s(z) \) is a rational function.

The following properties of the polylogarithm functions are standard and easy to show.
    \begin{enumerate}[label=(\arabic*)]
        \item \label{item: polylog 1} For all \( x \in \mathbb{R}\) and \( j \in \mathbb{Z},\) we have \( \Li_j'(x)=\Li_{j-1}(x)/x . \) 
        \item For all \( n  \geq 2,\)  \( \Li_{1-n}\) has exactly \( n \) distinct non-positive roots \( r_1^{(n)}=0>r_2^{(n)}>\dots > r_n^{(n)} = -\infty\) which satisfy the interlacement property
        \begin{equation*}
            0=r_1^{(n+1)} = r_1^{(n)}>r_2^{(n+1)}>r_2^{(n)}>\dots > r_n^{(n+1)} > r_n^{(n)} = r_{n+1}^{(n+1)} = -\infty.
        \end{equation*} 
        \item  For all \( n \geq 2 ,\) \( {\Li_{1-n}(x) < 0} \) for all \( x \in (r_2^{(n)},0).\)
        \item For all \( n \geq 3 ,\) \( \Li_{1-n}(x)>0\) for all \( x \in (r_3^{(n)},r_2^{(n)}).\) Hence, for any \( n \geq 4, \) since \( r_2^{(n-1)} \in (r_3^{(n)},r_2^{(n)}), \) we have \( \Li_{1-n}(r_2^{(n-1)})>0.\)\label{item: polylog 4}
        \item \( \lim_{n \to \infty} r_2^{(n)} = 0.\)
    \end{enumerate}  
One verifies that  \( r_2^{(3)}=-1, \) \(r_2^{(4)} = \sqrt{3}-2, \) and  \( r_2^{(5)} =2\sqrt{6}-5 .\)

In the next theorem, we describe what happens when \( m=2\) and \( x_2\) is close to \(1\).
    
\begin{theorem}\label{proposition: finite cases}
    Let \( \mathbf{x} = (x_1,x_2) \) be such that  \( 1= x_1 > x_2 \geq 0, \) \( q \in (0,1], \) and \( \pmb{\alpha} = (\alpha_1 , \alpha_2) \in (0,1)^2\) be such that \( \alpha_1+\alpha_2 = 1. \) Let \( X \sim \alpha_1 \Pi_{1-qx_1} + \alpha_2 \Pi_{1-qx_2}\) and \( n \geq 3.\) (Recall that by Theorem~\ref{proposition: stright lines}, for any \( n, \) \( X([n]) \) being in \( \mathcal{R}\) is independent of~\( q.\))
    \begin{enumerate}[label=(\alph*)] 
        
        \item Let \( k \in \{ 3, \dots, n\}.\)  
        \begin{itemize}
            \item If  \( \Li_{1-k}\bigl(-\alpha_1^{-1}(1-\alpha_1)\bigr) < 0, \) then \( \nu_n\bigl( [k] \bigr) >  0\) for \( x_2\) sufficiently close to \(1\).
            \item If \( \Li_{1-k}\bigl(-\alpha_1^{-1}(1-\alpha_1)\bigr) > 0, \) then \( \nu_n\bigl( [k] \bigr) <  0\) for \( x_2\) sufficiently close to \(1\).
        \end{itemize}
        (If \( \Li_{1-k}(-\alpha_1(1-\alpha_1)) = 0, \) then the sign of \( \nu_n([k])\) depends on which zero of \( \Li_{1-k}\) one is considering.)
        \label{item: general 1}
        
        \item  \( \alpha_1 \geq 1/(1- r_2^{(n)}) \) if and only if \( X([n]) \in \mathcal{R}\) for \( x_2 \) sufficiently close to \(1\) (see Figure~\ref{fig: finite case rays 3}).\label{item: general 2}

    \end{enumerate}
\end{theorem}

\begin{proof} 
    We first show that~\ref{item: general 1} holds.
    In this case, by~\eqref{eq: general probability}
    \begin{align*} 
        &\nu_n([k])
        =  \sum_{j=0}^k  (-1)^{k-j} 
        \binom{k}{j} \log \bigl(\alpha_1 + (1-\alpha_1) x_2^{n-j}\bigr)  
        \\&\qquad= \sum_{j=0}^k  (-1)^{k-j} 
        \binom{k}{j} \log \bigl( 1 +\alpha_1^{-1} (1-\alpha_1) x_2^{n-j})\bigr)  
        \\&\qquad = (-1)^{k+1} \sum_{j=0}^k  (-1)^{j} 
        \binom{k}{j} \Li_1\bigl(- \alpha^{-1}(1-\alpha_1) x_2^{n-j} \bigr). 
        \end{align*}
        One verifies that for any  \(\ell \in \{ 0,1,\dots, k+1\}, \) we have
        \begin{equation*}
            \sum_{j=0}^k  (-1)^{j} 
        \binom{k}{j}  (n-j)^\ell = \begin{cases}
            0 &\text{if } \ell < k \cr k!& \text{if } \ell=k\cr
            (n-k/2) (k+1)!  &\text{if } \ell = k+1.
        \end{cases}
        \end{equation*}
        Combining these observations and Property~\ref{item: polylog 1} of the polylogarithm functions, we obtain 
        \begin{align*}
             &\frac{d^\ell}{dx_2^\ell} \nu_n([k])|_{x_2=1} 
             \\&=
             \begin{cases}
                 0 &\text{if } \ell<k ,
                 \cr 
                 k! (-1)^{k+1} \Li_{1-k} \bigl(-\alpha_1^{-1}(1-\alpha_1)\bigr)\  &\text{if } \ell =k,
                 \cr
                 \frac{(k+1)! (-1)^{k+1} }{2}
    \Bigl(
    (2n-k)\Li_{-k} \bigl(-\alpha_1^{-1}(1-\alpha_1)\bigr)  
    -
    k
    \Li_{1-k}\bigl(-\alpha_1^{-1}(1-\alpha_1)\bigr) 
    \Bigr)
                 &\text{if } \ell=k+1.
             \end{cases} 
             \end{align*}
    (We will not need the case \( \ell=k+1 \) for~\ref{item: general 1}, but we will need it for~\ref{item: general 2} and hence include it here.) 
    Using a Taylor expansion of \( \nu_n([k]) \) around \( x_2=1, \) it follows that \(  \nu_n([k]) <0\) for \( x_2 \) near \(1\) if  
    \[
    0> \frac{d^k}{dx_2^k} \nu_n([k])\bigr|_{x_2=1} (-1)^k = -k! \Li_{1-k} \bigl(-\alpha_1^{-1}(1-\alpha_1)\bigr).
    \]
    Hence, if  
    \[
         \Li_{1-k} \bigl(-\alpha_1^{-1}(1-\alpha_1)\bigr)> 0,
    \]
    then 
    \(  \nu_n([k]) <0\) for \( x_2\) near \(1\).
    Analogously, if
    \[
         \Li_{1-k} \bigl(-\alpha_1^{-1}(1-\alpha_1)\bigr)< 0,
    \]
    then 
    \(  \nu_n([k]) >0\) for \( x_2\) near \(1\). This completes the proof of~\ref{item: general 1}. 

    We now show that~\ref{item: general 2} holds. 
    Note that  \( \alpha_1 < 1/(1-r_2^{(n)}) \Leftrightarrow -\alpha_1^{-1}(1-\alpha_1) \ < r_2^{(n)}=r_2^{(n)}.\) 
    If \( \alpha_1 > 1/(1-r_2^{(n)}), \) then \( \Li_{1-k}(-\alpha_1(1-\alpha_1)) <0\) for all \( k \in \{ 2,3, \dots, n\},\) and so, by~\ref{item: general 1}, we have \( \nu_n([k])>0\) for \( x_2\) close to \(1\) for all \( k \in \{ 1,2,\dots, n \}.\) Hence \( X([n]) \in \mathcal{R}\) for \( x_2\) close to \(1\). 
    Conversely, if \( \alpha_1 < 1/(1-r_2^{(n)}), \) then there is at least one \( k \in \{ 3,4,\dots, n\}\) such that \( -\alpha_1^{-1}(1-\alpha_1) \in (r_3^{(k)},r_2^{(k)}),\) and hence \( \Li_{1-k}(-\alpha_1(1-\alpha_1)) >0.\) For this \(k,\) by~\ref{item: general 1}, we have \( \nu_n([k])<0\) for \( x_2\) close to \(1\), and hence \( X([n]) \notin \mathcal{R}\) for \( x_2\) close to \(1\). 
    Finally, if \( \alpha_1 = 1/(1-r_2^{(n)}), \) then by the above, we have that  \( \nu_n([k])>0\) for \( x_2\) close to \(1\) for all \( k \in \{ 1,2,\dots, n-1 \}.\) Noting that for this \( \alpha_1,\) we have \( \Li_{1-n} (-\alpha_1^{-1}(1-\alpha_1))= 0 \) and \( \Li_{-n} \bigl(-\alpha_1^{-1}(1-\alpha_1)\bigr)> 0 ,\) and so we obtain \( \nu_n([n])>0\) for \( x_2\) close to \(1\) by using a Taylor expansion of \( \nu_n([n])\) around \( x_2=1\) of degree \( n+1\) and Property~\ref{item: polylog 4} of the polylogarithm functions.  Hence \( X([n]) \in \mathcal{R}\) for \( x_2\) close to \(1\) .
    Finally, if \( \alpha_1 = 1/(1-r_2^{(n)}), \) then by the above, we have that  \( \nu_n([k])>0\) for \( x_2\) close to \(1\) for all \( k \in \{ 1,2,\dots, n-1 \}.\) Noting that for this \( \alpha_1,\) we have \( \Li_{1-n} (-\alpha_1^{-1}(1-\alpha_1))= 0 \) and \( \Li_{-n} \bigl(-\alpha_1^{-1}(1-\alpha_1)\bigr)> 0 ,\) and so we obtain \( \nu_n([n])>0\) for \( x_2\) close to \(1\) by using a Taylor expansion of \( \nu_n([n])\) around \( x_2=1\) of degree \( n+1\) and Property~\ref{item: polylog 4} of the polylogarithm functions.  Hence \( X([n]) \in \mathcal{R}\) for \( x_2\) close to \(1\).
    Finally, if \( \alpha_1 = 1/(1-r_2^{(n)}), \) then by the above, we have that  \( \nu_n([k])>0\) for \( x_2\) close to \(1\) for all \( k \in \{ 1,2,\dots, n-1 \}.\) Noting that for this \( \alpha_1,\) we have \( \Li_{1-n} (-\alpha_1^{-1}(1-\alpha_1))= 0 \) and \( \Li_{-n} \bigl(-\alpha_1^{-1}(1-\alpha_1)\bigr)> 0 ,\) and so we obtain \( \nu_n([n])>0\) for \( x_2\) close to \(1\) by using a Taylor expansion of \( \nu_n([n])\) around \( x_2=1\) of degree \( n+1\) and Property~\ref{item: polylog 4} of the polylogarithm functions.  Hence \( X([n]) \in \mathcal{R}\) for \( x_2\) close to \(1\).
    This concludes the proof of~\ref{item: general 2}. 
\end{proof}

    \begin{figure}[htp]
    \centering
    \includegraphics[width=.7\linewidth,trim=50 0 0 0, clip]{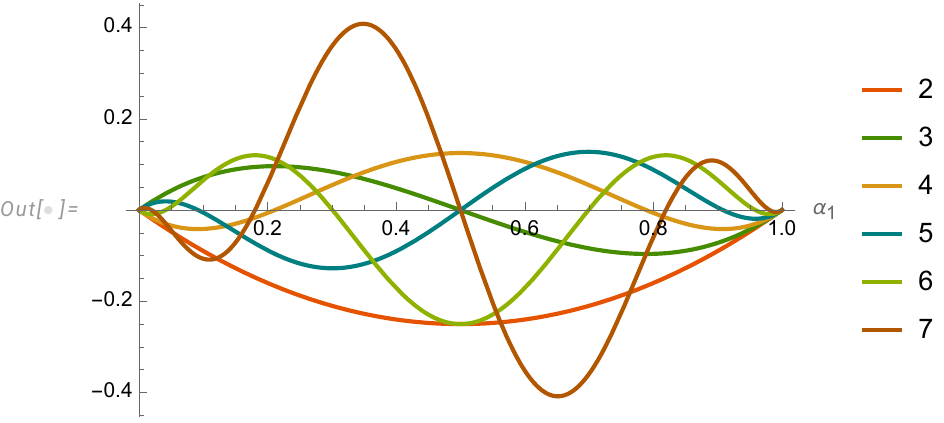}
    \caption{In the above figure, we draw \( \Li_{1-k}(-\alpha_1^{-1}(1-\alpha_1)) \) for \( k = 2,3,4,5,6,7 \) and \( \alpha_1 \in (0,1).\) Using Theorem~\ref{proposition: finite cases}, one verifies that when \( n=3,4,5,6,7 \) we get the following thresholds in \( \alpha_1\) for getting \( X([n]) \in \mathcal{R}\) when \( x_2\) is close to one: \( \frac{1}{2},\) \(\frac{1}{2} + \frac{1}{2\sqrt{3}} ,\) \( \frac{1}{2} + \frac{1}{\sqrt{6}} ,\) \(\frac{1}{2} +\frac{1}{2}\sqrt{\frac{\sqrt{105}+15}{30}},\) \( \frac{1}{2} + \frac{1}{2}\sqrt{\frac{\sqrt{15}+10}{15} }  \) (approximately equal to 0.5, 0.788675, 0.908248, 0.958684, 0.98085 respectively).}
    \label{fig: polylog}
\end{figure}

\section{Tree-indexed Markov chains and the Ising model}\label{sec: tree}

In this section, we apply results from the previous section to obtain results for tree-indexed Markov chains and the Ising model on \( \mathbb{Z}^d,\) \( d \geq 2.\)

Given a tree \( \mathcal{T} \) and parameters \( p,r \in (0,1), \) we construct a reversible tree-indexed \( \{ 0,1\}\)-valued Markov chain \( X \) indexed by \( V(\mathcal{T}) \) sequentially as follows. First, fix any root \( o \) of the tree and let \( X_o \sim r \delta_0 + (1-r) \delta_1. \) In the following steps, for any vertex \( j \in V(\mathcal{T})\) that is adjacent to some vertex \( i \in V(\mathcal{T}) \) which has already been determined, we let \( X_j \overset{d}{=} (1-p) \delta_{X_i} + p(r \delta_0 + (1-r) \delta_1),\) independently of everything else.
\color{black}{}

\begin{theorem}\label{theorem: MC on trees}
Let $d\ge 3$ and let \( \mathcal{T}\) be a tree where some vertex \( o \) has at least $d$ infinite disjoint paths emanating from it.  Let $X^{p,r}$ be the
    positively associated tree-indexed Markov chain on \( \mathcal{T}\) where \( p \) and \( r\) as in Remark~\ref{remark: p and r MC} belong to \( (0,1)\). 
    If  $r< 1/(1-r_2^{(d)}),$ where \( r_2^{(d)}\) is as in the paragraph before Theorem~\ref{proposition: finite cases},
        then \( X^{p,r} \notin \mathcal{R}.\)
 \end{theorem}
 
Note that since $r_2^{(d)}$ approaches \(0\) as $d\to \infty,$ we have that for any $r\in (0,1)$, there is some~$d$ such that \(X^{p,r} \notin \mathcal{R}\) for all \( p \in (0,1).\) \color{black}

\begin{proof}
    For \( k \geq 1, \) let  \( L_k = \{ \ell_1^{(k)},\ell_2^{(k)}, \dots, \ell_d^{(k)} \} \) be a set of vertices in \( \mathcal{T} \) such that for each \( j \in [d],\) 
    \begin{enumerate}[label=(\roman*)]
        \item \( \dist(\ell_j^{(k)},o) = k,\) and
        \item the paths between \( o\) and \( \ell_j^{(k)} ,\) \( j \in [d],\) are disjoint. 
    \end{enumerate} Note that conditioned on \( X^{p,r}_o,\) the random variables \( X^{p,r}_{\ell_1^{(k)}}, \dots, X^{p,r}_{\ell_d^{(k)}}\) are independent and identically distributed. Moreover, if we let
    \begin{equation*} 
        \begin{cases} 
            p_0^{(k)} \coloneqq P(X^{p,r}_{\ell_1^{(k)}} = 1 \mid X^{p,r}_o = 0) \cr
            p_1^{(k)} \coloneqq P(X^{p,r}_{\ell_1^{(k)}} = 1 \mid X^{p,r}_o = 1), 
        \end{cases}
    \end{equation*}
    then
    \begin{equation*}
        X|_{L_k} \sim P(X_o=0)  \Pi_{p_0^{(k)}}+P(X_o=1)  \Pi_{p_1^{(k)}}.
    \end{equation*}
    Note that \( p_0^{(k)}<p_1^{(k)}\) and that
    \begin{equation*}
        \lim_{k \to \infty}\frac{1-p_1^{(k)}}{1-p_0^{(k)}}  = 1.
    \end{equation*} 
    Using Theorem~\ref{proposition: finite cases}\ref{item: general 2}, it follows that for $r< \frac{1}{1-r_2^{(d)}}$
    and for any sufficiently large \( k,\) \( X|_{L_k} \notin \mathcal{R},\) and hence, by Lemma~\ref{lemma: finite to infinite}\ref{item: finite to infinite i}, \( X \notin \mathcal{R}.\)  
    \end{proof}

\begin{remark}
    Interestingly, we can show that if the tree \( \mathcal{T} \) consists of one vertex with three infinite rays emanating from the vertex, then, for \( r \) close enough to \( 1, \) one has that for all \( p \in (0,1),\) the Markov chain indexed by \( \mathcal{T}\) belongs to \( \mathcal{R}.\) In light of Theorem~\ref{theorem: MC on trees}, this tree therefore exhibits a phase transition in \( r,\) which we do not believe holds for the binary tree. This will be elaborated on and extended in future work.
    In fact, one can check that for the tree \( \mathcal{T}\) with four vertices and three leaves, the set of \( (p,r)\) for which the tree-indexed Markov chain on \( T\) is not in \( \mathcal{R}\) is a non-trivial subset of~\( (0,1)^2.\)
\end{remark}

\begin{theorem}\label{theorem: Ising}
    Let \( X\) be the Ising model on \( \mathbb{Z}^d ,\) \( d \geq 2,\) (identifying \( -1 \) with \( 0\)), and coupling constant \( J>0\). Then, if 
    $$
\frac{e^{2dJ}}{e^{2dJ}+ e^{-2dJ}} \le \frac{1}{1-r_2^{(2d)}}
$$
where \( r_2^{(2d)}\) is as in the paragraph before Theorem~\ref{proposition: finite cases}, we have that \( X \notin \mathcal{R}.\)
\end{theorem}

\begin{proof}

        Let \( o \coloneqq (0,0,\ldots,0) \) and %
    for $k\ge 1$, let \[
    L_k \coloneqq \{ (\pm k,0,\ldots,0),
    (0,\pm k,0,\ldots,0),\ldots,
    (0,0,\ldots,0, \pm k) \}. \]
    Next, let $C_k$ consist of o and the $2d$ direct paths
    from o to the points in $L_k$ and finally,
    let \( X^{(k)} \) be the restriction of \( X\) conditioned to be zero on the set \( \mathbb{Z}^2 \backslash C_k,\) to \( C_k.\)  Then, conditioned on \( X_o^{(k)}\) the random variables \( \{X^{(k)}_\ell\}_{\ell \in L_k }\) are independent and identically distributed.
    Moreover, if we, for \( \ell \in L_k\), let
    \begin{equation*} 
        \begin{cases} 
            p_0^{(k)} \coloneqq P(X^{(k)}_{\ell} = 1 \mid X_o^{(k)} = 0) = P\bigl(X_{\ell } = 1 \mid X_o = 0,\, X(\mathbb{Z}^2 \backslash C_k) \equiv 0\bigr)\cr
            p_1^{(k)}\coloneqq P(X^{(k)}_{\ell } = 1 \mid X_o^{(k)} = 1)= P\bigl(X_{\ell } = 1 \mid X_o = 1,\, X(\mathbb{Z}^2 \backslash C_k) \equiv 0\bigr)
        \end{cases}
    \end{equation*}
    then
    \begin{equation*}
        X^{(k)}|_{L_k} \sim P(X_o^{(k)}=0)  \Pi_{p_0^{(k)}}+P(X_o^{(k)}=1)  \Pi_{p_1^{(k)}}.
    \end{equation*}
    Note that since \( J>0, \) for every \( k\geq 1, \) \( p_0^{(k)} \neq p_1^{(k)},\) and moreover
    \begin{equation*}
        \lim_{k \to \infty}\frac{1-p_1^{(k)}}{1-p_0^{(k)}} =1.
    \end{equation*} 
     Finally, observe that \(P(X_o^{(k)}=0)\le 
    \frac{e^{2dJ}}{e^{2dJ}+ e^{-2dJ}}\).
    Using Theorem~\ref{proposition: finite cases}\ref{item: general 2},
    it follows that for any sufficiently large \( k,\) \( X^{(k)}|_{L_k} \notin \mathcal{R},\) and hence, by Lemma~\ref{lemma: finite to infinite}\ref{item: finite to infinite i}~and~\ref{item: finite to infinite iii}, \( X \notin \mathcal{R}.\)
\end{proof}

\begin{remark}
    For $d=2$, if we look at the critical value $J_c=1/2 \log(1 + \sqrt{2}) =0.440687$, one has that $\frac{e^{4J_c}}{e^{4J_c}+ e^{-4J_c}}  > \frac{3+ \sqrt{3}}{6} =\frac{1}{1-r_2^{(4)}}$ and hence the above theorem is only applicable  within a subset of the subcritical regime  (although we have no suspicions whatsoever that the conclusion fails somewhere). For  higher dimensions,  it is known that $J_c(d)\asymp \frac{1}{d}$ as $d\to \infty$ and since $r_2^{(2d)}$ goes to 0 as $d\to \infty$, we can conclude that for all sufficiently high dimensions,  the above theorem rules out being in  $\mathcal{R}$ throughout the subcritical regime and partly into the supercritical regime.   
    The above result can with small modifications be extended to the Ising model with an external field.
\end{remark}

\section{The stationary case}\label{sec: stationary}

In this section, we consider the 
stationary case and hence any \( \nu\) that we are considering will be assumed to be translation invariant.\color{black}

After having obtained the results in this section, we learned from
Nachi Avraham-Re'em and Michael Bj\"orklund that most of the results in this section follow from known results in the theory of
so-called Poisson suspensions, an area within ergodic theory and primarily within infinite measure ergodic theory. See for example~\cite{roy1,roy2}.
Since this section is only four pages, we decided to leave it as is, providing fairly direct proofs of the results stated, in the spirit of the rest of the paper,
rather than introducing the notion of a Poisson suspension and refer to the relevant theorems in the literature. For example, while Theorem~\ref{theorem: bernoulli shift} would follow from the known and easy result  that ``Poisson suspensions of dissipative systems are Bernoulli'', the proof here gives an easy factor map from an i.i.d.\ process to our system.

We begin by giving both sufficient and necessary conditions on $\nu$ for $X^\nu$ to be ergodic. For simplicity, we stick to one-dimensional processes.

We start with the simplest necessary condition as a warm-up.

\begin{proposition}\label{proposition: finite ergodic}
    Assume that $\nu$ is finite and not the zero measure. Then $X^\nu$ is not ergodic.
\end{proposition}

\begin{proof}
	This follows immediately from the fact that \( P(X^\nu \equiv 0) = e^{-\| \nu \|} \in (0,1). \)\color{black}
	%
\end{proof}

\begin{remark}
    If we let $\mu_k$ be the law of $\max\{Y^1,Y^2,\ldots, Y^k\}$ for \(k \geq 1\) and \(\mu_0\) the Dirac measure on the sequence consisting only of \(0\)'s, the distribution of $X^\nu$ is  given by 
    $$
        \sum_{k=0}^\infty \frac{ \|\nu\|^ke^{-\|\nu\|}}{k!} \mu_k .
    $$
    If $Y$ is ergodic, then the above will often, but not always, yield the ergodic decomposition of $X^\nu$. Reasons why it does not in general are
    \begin{enumerate}[label=(\arabic*)]
        \item if $Y \equiv 1$  a.s., then (and only then) $\mu_1=\mu_k$ for all $k\ge 1,$ and 
        \item it is possible that $\mu_k$ is not ergodic and must be further decomposed but this cannot happen if $Y$ is weak-mixing.
    \end{enumerate} 
\end{remark}

The next theorem strengthens  Proposition~\ref{proposition: finite ergodic} and provides two equivalent conditions for ergodicity. Before stating the theorem, we define
$$
    \mathcal{Z} \coloneqq \Bigl\{\eta \in \{0,1\}^{\mathbb{Z}}\colon \lim_{n\to\infty} \frac{\sum_{k=0}^{n-1}\eta_k}{n} =0\Bigr\}.
$$
In other words, $\mathcal{Z}$ consists of the configurations 
with zero density. It is immediate that $\mathcal{Z}$ is a translation invariant set.

\begin{theorem}\label{theorem: main ergodic}
    Assume that $P(X^{\nu}\equiv 1) <1,$ or equivalently, that $\nu(\mathcal{S}_{0})< \infty$.
    Then the following are equivalent.
    \begin{enumerate}[label=(\roman*)]
        \item $\nu$ cannot be expressed as $\nu_1+\nu_2$ where $\nu_1$ and $\nu_2$ are translation invariant and $\nu_2$ is a nonzero finite measure 
        \item $\nu(\mathcal{Z}^c)=0$ 
        \item $X^{\nu}$ is ergodic.
    \end{enumerate}
\end{theorem}

\begin{proof} 
    (iii) implies (i). \\
    Assume that $\nu=\nu_1+\nu_2$ where $\nu_1$ and $\nu_2$ are translation invariant and $\nu_2$ is a nonzero finite measure.
    By Lemma~\ref{lemma: 2.10new}, if we let \( X^{\nu_1} \) and \( X^{\nu_2} \) be independent, then $X^\nu$ and $\max\{X^{\nu_1},X^{\nu_2}\}$ have the same distribution.
    Since $\nu_2$ is a finite measure, $X^{\nu_2} \equiv 0$  with positive probability, and hence the density of  $X^\nu$ is the same as the density  of $X^{\nu_1}$ with positive probability. Since $\nu_2$ is nonzero, $X^{\nu_2}$ has a positive density with positive probability and since $X^{\nu_1}$ and $X^{\nu_2}$ are independent and $P(X^{\nu_1}\equiv 1) <1$, the density of $X^\nu$ is strictly larger than the density of $X^{\nu_1}$ with positive probability. Together this implies that $X^{\nu}$ cannot be ergodic.

    (i) implies (ii). \\
    It suffices to show that for any $\delta >0$,
    $$
        \nu\Bigl(\bigl\{\eta \in \{0,1\}^\mathbb{Z}\colon \limsup_{n\to\infty} \frac{\sum_{k=0}^{n-1}\eta_k}{n} \ge \delta \bigr\}\Bigr)=0.
    $$
    To do this, we will show that $\nu(A_\delta)=0$ where
    $$
        A_\delta \coloneqq
        \bigl\{\eta \in \{0,1\}^\mathbb{Z}\colon \liminf_{n\to\infty} \frac{\sum_{k=0}^{n-1}\eta_k}{n} \ge \delta \bigr\},\quad \delta > 0
    $$
    and then apply~\cite[Theorem 2]{Satoerg} to conclude the previous statement. 
    Using Fatou's Lemma, we get
    $$
        \delta\nu(A_\delta) \le
        \int  \liminf_{n\to\infty} 
        \frac{\sum_{k=0}^{n-1}\eta_k}{n} \, d\nu|_{A_\delta}
        \le 
        \liminf_{n\to\infty} \int 
         \frac{\sum_{k=0}^{n-1}\eta_k}{n}\ d\nu|_{A_\delta} =
        \nu|_{A_\delta}(\mathcal{S}_{0}) < \infty.
    $$
    Hence $\nu(A_\delta)<\infty,$ and so, by assumption, we can conclude that $\nu(A_\delta)=0$ as follows. Since $A_\delta$ is translation invariant, we can write $\nu=\nu|_{A_\delta}+ (\nu-\nu|_{A_\delta})$ where the two summands are translation invariant. Using the assumption (i), the desired conclusion follows.
 
    (ii) implies (iii). \\
    We first show that
    \begin{equation}\label{eq: first step}
        \lim_{n\to\infty} 
        \frac{1}{n} 
        \sum_{k=0}^{n-1}P\bigl(X^\nu_0=X^\nu_k=0\bigr)=
        P\bigl(X^\nu_0=0\bigr)^2.
    \end{equation}
    By the first statement in Proposition~\ref{prop: correlation}, this is equivalent to showing that 
    $$
        \lim_{n\to\infty} 
        P\bigl(X^\nu_0=0\bigr)^2 \cdot 
        \frac{1}{n} 
        \sum_{k=0}^{n-1} (e^{\nu(\mathcal{S}_{0}\cap \mathcal{S}_{k})}-1)
        =0.
    $$ 
    Since \( P(X^\nu \equiv 1)<1,\) we have
    \begin{equation*}
        \sup_k \nu(\mathcal{S}_{0}\cap \mathcal{S}_{k}) \leq \nu(\mathcal{S}_{0})<\infty,
    \end{equation*}
    and hence there exists $c>0$ so that for all $n \in \mathbb{N}$ we have
    \begin{equation}\label{eq:lest eq of proof}
        \frac{1}{n}  \sum_{k=0}^{n-1} (e^{\nu(\mathcal{S}_{0}\cap \mathcal{S}_{k})}-1)
        \leq 
        \frac{c}{n} 
        \sum_{k=0}^{n-1} \nu(\mathcal{S}_{0}\cap \mathcal{S}_{k})
        =
        c \int \mathbf{1}_{\mathcal{S}_{0}}\frac{1}{n} 
        \sum_{k=0}^{n-1} \mathbf{1}_{\mathcal{S}_{k}}d\nu
    \end{equation}
    By our key assumption, the random sum \( \frac{1}{n} \sum_{k=0}^{n-1} \mathbf{1}_{\mathcal{S}_{k}}\) approaches 0 $\nu$-a.e. Since $\nu(\mathcal{S}_{0})<\infty$ we can now apply the Bounded Convergence Theorem to conclude that the right hand side of~\eqref{eq:lest eq of proof} approaches 0 as $n\to\infty$. This shows that~\eqref{eq: first step} holds.

    From here, there are two ways to complete the proof. One quick way to do this is that one can observe that Lemma~\ref{lemma: 2.10new} implies that $X^\nu$ is max-infinitely divisible (see~\cite{ks} for the definition) and then apply Theorem 1.2 from this latter reference. 
    However, one can proceed directly as follows. Let $S_1$ and  $S_2$ be two finite subsets of $\mathbb{Z}.$ Let $A$ be the event that $X^\nu(S_1) \equiv 0$  and $B$ be the event that $X^\nu(S_2)\equiv 0.$
    One can show, analogously to the first step, that 
    \begin{equation}\label{eq: cylinder sets}
        \lim_{n\to\infty} 
        \frac{1}{n} 
        \sum_{k=0}^{n-1}P(A\cap T^{-k} B) =P(A)P(B)
    \end{equation}
    where $T^{-k}$ is a \( k\) - step left shift. One does this by applying the full statement of Proposition~\ref{prop: correlation} and then easily modifying the argument above.
    Using inclusion-exclusion, one can conclude~\eqref{eq: cylinder sets} for any cylinder sets \( A\) and \( B.\) This yields the ergodicity.
\end{proof} 

\begin{remark}
    It is easy to see that for positively associated processes, ergodicity implies weak-mixing. Alternatively,  this is stated in~\cite[Theorem 1.2]{ks} for max-infinitely divisible processes. In addition, by~\cite[Theorem 1.1]{ks}, to prove mixing, it suffices to prove decaying pairwise correlations which amounts to showing that $\lim_{n\to \infty} \nu(\mathcal{S}_{\{0,n\}}^\cap)=0$.
\end{remark}

We next give a sufficient condition on $\nu$ which yields much stronger ergodic behavior.

\begin{theorem}\label{theorem: bernoulli shift}
    Assume that
    $\nu$ is such that  $\nu$-a.e.\color{black}{} $S \in \mathcal{P}(\mathbb{Z})$ has a smallest element. Then $X^\nu$ is a Bernoulli Shift; i.e.\ it is a factor of i.i.d.'s.
\end{theorem}
 
\begin{proof}
    Let $\mathcal{T}_k$ be the set of subsets whose smallest element is $k$. Then $\nu$ is concentrated on $\bigcup \mathcal{T}_k$ and by translation invariance, the sets $\mathcal{T}_k$ are disjoint and all have the same $\nu$-measure.
    
    For the rest of the proof, the three cases \( \nu(\mathcal{T}_k) = \infty, \) \( \nu(\mathcal{T}_k) = 0, \) and \( \nu(\mathcal{T}_k) \in (0, \infty) \) will be treated separetedly.

\begin{description}
	\item[Case 1:] Assume that $\nu(\mathcal{T}_k) = \infty$ for some (and hence all) $k$. 
	
	In this case, it is easy to see that $X^\nu\equiv 1$ a.s., and hence we are done.
	
	\item[Case 2:] Assume that $\nu(\mathcal{T}_k) = 0$  for some (and hence all) $k$.

    In this case, $\nu$ is the zero measure and $X^\nu\equiv 0$ a.s. and hence, we are done.
    
    \item[Case 3:] Assume that $\nu(\mathcal{T}_k) \in (0,\infty)$ for some (and hence all) \( k.\)

    We will represent $X^\nu$ as a translation invariant function of random variables $\{U_n\}_{n\in \mathbb{Z}}$, each uniform on the unit interval $[0,1]$ as follows. For each $k$, we can use $U_k$ to generate a Poisson point process $\mathcal{P}_k$ on $\mathcal{T}_k$ with intensity measure $\nu_{\mathcal{T}_k}$. Then, letting $X_n=1$ if and only if $n\in \cup_{k\le n}\mathcal{P}_k$, we have that $X$ has the same distribution as $X^\nu$ and we are done.
\end{description}

\end{proof}

\begin{remarks}\mbox{}
    \begin{enumerate}
        \item If there is a uniform bound on the radius of the  subsets where $\nu$ is supported, then the above of course  holds and the  map is a block map. 
        \item If there is a uniform bound only on the sizes of the  subsets where $\nu$ is supported (a weaker assumption), then the above of  course holds but this is not necessarily a block map. In fact, if there is not a uniform bound  on the radius of the  subsets where $\nu$ is supported, then the above map is not finitary. 
        \item Much of the above extends to $\mathbb{Z}^d$.
    \end{enumerate}
\end{remarks}

\begin{theorem}\label{theorem: finite sets bernoulli}
    Assume for $\mathbb{Z}^d$ that $\nu$ is supported on sets $S$ for which there exists $(k_1,\ldots,k_d)$ so that $i_j\ge k_j$ for each $j$ for all $(i_1,\ldots,i_d)\in S$. Then $X^\nu$ is a Bernoulli Shift; i.e.\ it is a factor of i.i.d.s. In particular, if $\nu$ is supported on finite sets, then $X^\nu$ is a Bernoulli Shift.
\end{theorem}

\begin{remark}
    The assumption loosely says that $\nu$ is concentrated on sets which are contained inside of a "north-east" quadrant.
\end{remark} 

\begin{proof}[Proof of Theorem~\ref{theorem: finite sets bernoulli}]
We only outline the proof since it requires only minor 
modifications
of the previous proof. We also only do this for $d=2$.
Let $\mathcal{T}_{k_1,k_2}$ be the set of subsets $\mathbb{Z}^d$ where 
$i_1\ge k_1$ and $i_2\ge k_2$ for all $(i_1,i_2)\in S$ and 
$(k_1,k_2)$ is maximal
with respect to the lexographic order with this property. We 
only deal
with the non-trivial case where $\nu(\mathcal{T}_{k_1,k_2})\in (0,\infty)$.
We will represent $X$ as a translation invariant function of random variables 
$\{U_{k_1,k_2}\}$, uniform on the interval $[0,1]$. For each $(k_1,k_2)$, we can use $\{U_{k_1,k_2}\}$
to generate a Poisson point process on $\mathcal{T}_{k_1,k_2}$. Then 
one proceeds as in the 1-dimensional case.
\end{proof}

In view of the results in this section, it becomes clear
that, from an ergodic theoretic point of view, the most
interesting case is when $\nu$ is supported on infinite
zero density sets. This is analogous to what happened
when one studied divide and color models in~\cite{st}.

The following provides us with such an example.

\begin{example}
    Let \( X \) be the random interlacements process on \( \mathbb{Z}^3.\) Then, by construction, \( X = {X}^{\nu},\) where \( \nu\) is the corresponding measure with support on the set of all simple random walks trajectories, transient in both directions,
    in \( \mathbb{Z}^3,\) and with the property that the mass assigned to the set of all simple random walks that intersects a given finite set is finite and non-zero. 
    Let $\nu'$ be $\nu$ restricted to the $x$-axis. 
    Since a random walk on \( \mathbb{Z}^2\) is recurrent, 
    $\nu'$ is supported on infinite 0 density sets.
    Letting \( X'\) be the restriction of \( X\) to 
    the \( x \)-axis, we of course have \( X' = {X}^{\nu'}\).
By Theorem~\ref{theorem: main ergodic}, $X'$ is ergodic.
In fact, in~\cite{bb}, the full interlacement process $X$ was
shown to be a Bernoulli shift which implies that 
$X'$ is also a Bernoulli shift.
\end{example}

\section{Questions}\label{sec: questions}

In this section, we collect a few interesting and natural questions about Poisson representable processes.
\color{black}

\begin{enumerate}[label= Q\arabic*]

\item 
Is the subcritical (high temperature) Curie-Weiss model in
$\mathcal{R}$ for large $n$?

\item Is the Ising model in \( \mathbb{Z}^d\) for $J>0$
always not in $\mathcal{R}$?

\item Is there any nontrivial tree-indexed Markov chain on an infinite regular tree (other than \( \mathbb{Z}\)) that is in $\mathcal{R}$? \label{item: question 3}

\item 
Take i.i.d.\ bond percolation on \( \mathbb{Z}^d,\) \( d>1, \) \( p > p_c.\) Take \( X(v) = 1\) if $v$ belongs to the infinite
component and \( 0\) otherwise. Is $X\in \mathcal{R}$?
Note, it is known (see~\cite{BHK}) that the law of \( X\) is 
downward FKG. If $X=X^\nu$ for some $\nu$, it is easy to see
that \( \nu\) must then be supported on sets that have no finite components.

\item Is the upper invariant measure for the contact process in $\mathcal{R}$? It is known to be downwards FKG.\label{item: question 5}

\item Given $\nu$, can the behavior of 
$X^{c\nu}$ depend in an essential way on
$c$? This is clearly a fairly vague question;
one example where one sees this kind
of phenomenon is in the percolation
properties of the interlacement process.

\item 
We have seen that if 
   $\nu= \sum_{i\in \mathbb{Z}, n\ge 1} a_n \delta_{i, i+n}$ with
$\sum a_n < \infty$, then $X^\nu$ is a factor of i.i.d.'s.
   Is $X^\nu$ in fact a finitary factor of i.i.d.'s?

\item Related to Proposition~\ref{proposition: domination from below}, what can one say concerning the class of product measures which dominate \( X^\nu\) for a translation invariant measure \( \nu\) on \( \mathcal{P}(\mathbb{Z}^d)\smallsetminus \{ \emptyset \}\)?
\end{enumerate}

\begin{remark}
	Since the submission of this manuscript, some of the questions listed above have been partially answered. In particular, in~\cite{bf2025}, it was shown that for some parameter choices, the contact process is not in \( \mathcal{R}, \) thus partially answering~\ref{item: question 5}, and in~\cite{f2025}, tree indexed Markov processes were studied in more detail to partially answer~\ref{item: question 3} in the negative.	
\end{remark}

\addcontentsline{toc}{section}{Acknowledgements}
\paragraph{Acknowledgements}

Our foremost thanks go to Barney, without whom, this project would almost certainly not have been completed or might have taken an additional hundred years. In addition, there are many people we consulted for questions, references, and discussion. For these, we thank David Aldous, Nachi Avraham-Re'em, Stein Andreas Bethuelsen, Michael Bj\"orklund, Harry Crane, Persi Diaconis, Svante Janson, Bob Pego, Konrad Schm\"udgen, Jan Stochel, Jerzy Stochel, and Aernout van Enter. We especially acknowledge Christian Berg and Svante Janson, whose responses to a question about moment sequences led us to a better understanding of the permutation invariant case.  We also thank the two anonymous referees for many useful comments and for simplifying our proof of Proposition~\ref{proposition: finite ergodic}.\color{black}{} Lastly, the initial idea of this project stemmed from attending Johan Tykesson's excellent lectures on the interlacement process.

The second author was supported as a jubilee professor at Chalmers University of Technology. 
The third author acknowledges the support of the Swedish Research Council, grant no. 2020-03763.


\begin{thebibliography}{99}


	
    \bibitem{BHK} van den Berg, J., H\"{a}ggstr\"{o}m, O., Kahn, J. Some conditional correlation inequalities for percolation and related processes. Random Struct. Algorithms 29 (2006), no.4, 417--435.

	\bibitem{bf2025} Bethuelsen, S. A., Forsstr\"om, M. P., Mixing for Poisson representable processes and consequences for the Ising model and the contact process, preprint available at \url{https://arxiv.org/abs/2501.14445} (2025)

    \bibitem{bl} Bloznelis, M., Leskel\"a, L. Clustering and percolation on superpositions of Bernoulli random graphs. Random Struct. Algorithms (2023), 283--342.
	
    \bibitem{BJR}
    Bollob\'as, B., Janson, S., Riordan, O. Sparse random graphs with clustering, Random Struct. Algorithms 38(3) (2011), 269--323.

    \bibitem{bb} Borb\'enyi, M., R\'ath, B., Rokob, S. Random interlacement is a factor of i.i.d., Electron. J. Probab. 28 (2023).

    \bibitem{crane} Crane, H. Combinatorial L\'evy processes, Ann. Appl. Probab. 28(1) (2018), 285--339.  
    
    \bibitem{DRS} Drewitz, A., R\'a{}th, B., Sapozhnikov, A. An introduction to random interlacements. Springer Briefs Math. Springer Cham, (2014). x+120 pp.
    
	\bibitem{fds} Fishburn, P. C., Doyle, P. G., Shepp, L. A., The Match Set of a Random Permutation Has the FKG Property, The Annals of Probability, Vol. 16, No. 3 (1988), 1194--1214.
    
    
	\bibitem{f2025} Forsstr\"om, M. P., Poisson representations for tree-indexed Markov chains, preprint available at \url{https://arxiv.org/abs/2501.14428} (2025) 
	
	\bibitem{g2023} Gladkov, N., A strong FKG inequality for multiple events, preprint available at \url{https://arxiv.org/abs/2305.02653} (2023)
	
    \bibitem{ks}
    Kabluchko, Z., Schlather, M.
    Ergodic properties of max-infinitely divisible processes, Stoch. Process. Their Appl., 120 (2010), 281--295.

    \bibitem{Kahn} Kahn, J. A note on positive association, preprint, arXiv:2210.08653 (2022). 

    \bibitem{lp} Last, G.,  Penrose, M. Lectures on the Poisson Process (Institute of Mathematical Statistics Textbooks). Cambridge University Press. (2017)

    \bibitem{LawlerTrujilloFerreras}
    Lawler, G., Trujillo Ferreras, J. A., Random walk loop soup,
Trans. Amer. Math. Soc. 359, no.2, (2007) 767--787.

    \bibitem{L} Liggett, T.M. Survival and coexistence in interacting particle systems, NATO Adv. Sci. Inst. Ser. C: Math. Phys. Sci., 420, Kluwer Academic Publishers Group, Dordrecht (1994), 209--226.

    \bibitem{ls} Liggett, T.M., Steif, J.E. Stochastic domination: the contact process, Ising models and FKG measures, Annales de l'IHP Probabilit\'es et statistiques 42 (2), 40(2006), 223--243.

    \bibitem{roy1} Roy, E. Ergodic properties of Poissonian ID processes, Ann. Probab. 35 (2007), no. 2, 551--576.

    \bibitem{roy2} Roy, E. Poisson suspensions and infinite ergodic theory. Ergod. Th. \& Dynam. Sys. 29(2) (2009), 667--683.

    \bibitem{Satoerg} Sato, R. On a decomposition of transformations in infinite measure spaces, Pac. J. Math., Vol. 44, No. 2 (1973), 733--738.
    
     \bibitem{st} Steif, J.E., Tykesson, J. Generalized divide and color models, ALEA, Lat. Am. J. Probab. Math. Stat. 16 (2019), 1--57.
     
    \bibitem{sh} Steutel, F.W., van Harn, K. Infinite Divisibility of Probability Distributions on the Real Line (1st ed.) (2003). CRC Press. 

\end{thebibliography}
\end{document}